\documentclass[a4paper, 12pt]{amsart}
\usepackage{amsmath}
\usepackage{amscd}
\usepackage{amssymb}
\usepackage{amsthm}
\usepackage[dvipdfmx]{graphicx, xcolor, pdfpages}

\newtheorem{lemma}{{\sc Lemma}}[section]
\newtheorem{corollary}[lemma]{{\sc Corollary}}
\newtheorem{proposition}[lemma]{{\sc Proposition}}
\newtheorem{theorem}[lemma]{{\sc Theorem}}
\newtheorem{remark}[lemma]{{\sc Remark}}

\newtheorem{example}[lemma]{{\sc Example}}

\numberwithin{equation}{section}

\def\Ga{{\mathfrak{a}}}
\def\Gb{{\mathfrak{b}}}

\def\Gg{{\mathfrak{g}}}
\def\Gh{{\mathfrak{h}}}

\def\Gn{{\mathfrak{n}}}

\def\GA{{\mathfrak{A}}}
\def\GB{{\mathfrak{B}}}
\def\BA{{\mathbf{A}}}

\def\BC{{\mathbf{C}}}

\def\BF{{\mathbf{F}}}

\def\BQ{{\mathbf{Q}}}
\def\BZ{{\mathbf{Z}}}

\def\CB{{\mathcal B}}
\def\CC{{\mathcal C}}

\def\CF{{\mathcal F}}
\def\CH{{\mathcal H}}
\def\CO{{\mathcal O}}
\def\CI{{\mathcal I}}

\def\CP{{\mathcal P}}
\def\CR{{\mathcal R}}
\def\CS{{\mathcal S}}
\def\CT{{\mathcal T}}

\def\alg{{\rm alg}}

\def\Comod{\mathop{\rm Comod}\nolimits}

\def\End{\mathop{\rm{End}}\nolimits}

\def\gr{{\mathop{\rm{gr}}\nolimits}}
\def\Hom{\mathop{\rm Hom}\nolimits}

\def\id{\mathop{\rm id}\nolimits}

\def\Ind{\mathop{\rm Ind}\nolimits}
\def\inte{{\mathop{\rm{int}}\nolimits}}
\def\Image{\mathop{\rm Im}\nolimits}
\def\irr{{\rm irr}}
\def\Ker{\mathop{\rm Ker\hskip.5pt}\nolimits}

\def\Mod{\mathop{\rm Mod}\nolimits}

\def\op{{\mathop{\rm op}\nolimits}}

\def\Proj{\mathop{\rm Proj}\nolimits}

\def\res{\mathop{\rm res}\nolimits}
\def\Res{\operatorname{Res}\nolimits}
\def\Spec{{\rm{Spec}}}

\def\Tor{{\rm{Tor}}}

\begin{document}
\title[Affine open covering of the quantized flag manifolds]{Affine open covering of the quantized flag manifolds at roots of unity}
\author{Toshiyuki TANISAKI}
\address{
9-3-12 Jiyugaoka, Munakata, Fukuoka, 811-4163 Japan}
\email{ttanisaki@icloud.com}
\begin{abstract}
We show that the quantized flag manifold at a root of unity has natural affine open covering parametrized by the elements of the Weyl group.
In particular, the quantized flag manifold turns out to be a quasi-scheme in the sense of Rosenberg \cite{R}.
\end{abstract}
\maketitle

\section{Introduction}
Let $G$ be a connected semisimple algebraic group over the complex number field $\BC$, 
and let
$B$, $B^+$ be Borel subgroups of $G$ such that $B\cap B^+$ is a maximal torus of $G$.
The homogeneous space $\CB=G/B$ is a projective algebraic variety called the flag manifold.
Let $W$ be the Weyl group of $G$.
We have an affine open covering 
$\CB=\bigcup_{w\in W}\CB^w$, where
$\CB^w=wB^+B/B$.
Let $\CR$ be the homogeneous coordinate algebra of $\CB$, and let $\CR^w$ be the coordinate algebra of $\CB^w$ so that 
\[
\CB=\Proj\,\CR, 
\qquad
\CB^w=\Spec\, \CR^w\quad(w\in W).
\]

Let us consider the situation where $G$ is replaced by the corresponding quantum group.
Let $K$ be a field equipped with $q\in K^\times$.
Using the quantum group we can naturally define $q$-analogues $\CR_{K,q}$, $\CR^w_{K,q}$ of 
$\CR$, $\CR^w$ respectively.
Here, $\CR_{K,q}$ is  a graded $K$-algebra, and  $\CR^w_{K,q}$ is a $K$-algebra.
We have $\CR_{\BC,1}\cong\CR$ and $\CR^w_{\BC,1}\cong\CR^w$.
A major difference compared to the ordinary case $q=1$  is the fact that  $\CR_{K,q}$ and  $\CR^w_{K,q}$ are non-commutative in general.
Hence  in order to understand the ``quantized flag manifold'' in a geometric manner,
we need the language of non-commutative algebraic geometry, which has been developed 
by Artin-Zhang \cite{AZ}, Ver\"{e}vkin \cite{V}, Rosenberg \cite{R}
following Manin's idea \cite{M}.
Using $\CR_{K,q}$ and  $\CR^w_{K,q}$, 
we can define
as in Rosenberg \cite{R}, Lunts-Rosenberg \cite{LR} (see also Joseph \cite{Jo0}) the  abelian categories 
\begin{equation}
\label{eq:qcoh}
\Mod(\CO_{\CB_{K,q}}),\qquad
\Mod(\CO_{\CB^w_{K,q}})\quad (w\in W),
\end{equation}
which are regarded as the categories of ``quasi-coherent sheaves'' on the virtual spaces
\[
\CB_{K,q}=\Proj\,\CR_{K,q}, 
\qquad
\CB^w_{K,q}=\Spec\, \CR^w_{K,q}\quad(w\in W)
\]
respectively, 
and the exact functors
\begin{equation}
\label{eq:inc}
(i^w_{K,q})^*:\Mod(\CO_{\CB_{K,q}})\to
\Mod(\CO_{\CB^w_{K,q}})
\qquad(w\in W).
\end{equation}
In order to verify that $\Mod(\CO_{\CB_{K,q}})$ defines a quasi-scheme $\CB_{K,q}$ in the sense of \cite{R} we need to show the patching property
\begin{equation}
\label{eq:patch}
M\in\Mod(\CO_{\CB_{K,q}}), \quad
(i^w_{K,q})^*M=0\;\;(\forall w\in W)\;\;
\Longrightarrow\;\; M=0.
\end{equation}
This holds for $q=1$ since
$\CB_{K,1}$ is isomorphic to the ordinary flag manifold over $K$.
From this we can derive the property \eqref{eq:patch} when $q$ is transcendental over the prime field $K_0$ of $K$ (see \cite{LR}).

The main result of this paper is \eqref{eq:patch} where $q$ is a root of unity.
By the aid of Lusztig's quantum Frobenius homomorphism
we can reduce its proof to the case where $q=\pm1$.
In the case $q=1$ \eqref{eq:patch} is a classically known fact as mentioned above.
The proof is rather involved in the case $q=-1$.
We will construct 
an isomorphism $\CR_{K,-1}\cong \CR_{K,1}$ of graded vector spaces.
Although this isomorphism does not preserve the ring structure, it satisfies some favorable properties so that
we can derive \eqref{eq:patch} for $q=-1$ from that for $q=1$.

\section{Quantized enveloping algebras}
\label{sect:QE}
\subsection{}
Let $G$ be a connected reductive algebraic group over $\BC$, and let $H$ be a maximal torus of $G$.
We denote by $(X,\Delta,Y,\Delta^\vee)$ the  root datum associated to $G$ and $H$.
Namely, 
\[
X=\Hom(H,\BC^\times), \qquad
Y=\Hom(\BC^\times, H),
\]
and $\Delta$ (resp.\ $\Delta^\vee$) is the set of roots (resp.\ coroots).
The coroot corresponding to $\alpha\in\Delta$ is denoted as $\alpha^\vee\in\Delta^\vee$.
We fix a set of simple roots $\{\alpha_i\mid i\in I\}$ of $\Delta$, and denote the corresponding set of positive roots by $\Delta^+$.
The Weyl group $W$ is the subgroup of $GL(\Gh^*)$ generated by the simple reflections
\[
s_i:\Gh^*\to\Gh^* \quad(
\lambda\mapsto
\lambda-\langle\lambda,\alpha_i^\vee\rangle\alpha_i)
\]
for $i\in I$.
Set
\begin{gather*}
Q=\sum_{\alpha\in\Delta}\BZ\alpha\subset X,
\qquad
Q^\vee=\sum_{\alpha\in\Delta}\BZ\alpha^\vee\subset Y,
\\
Q^+=
\sum_{\alpha\in\Delta^+}\BZ_{\geqq0}\alpha\subset Q,
\qquad
X^+=
\{\lambda\in X\mid
\langle\lambda,\alpha_i^\vee\rangle\geqq0\;\;(i\in I)\}
\subset X.
\end{gather*}
For $i, j\in I$ we set $a_{ij}=\langle\alpha_j,\alpha_i^\vee\rangle$.

Let $\Gg$ and $\Gh$ be the Lie algebras of $G$ and $H$ respectively.
We will identify $X$ (resp.\ $Y$) with a $\BZ$-lattice of $\Gh^*$ (resp.\ $\Gh$).
We have the root space decomposition
\[
\Gg=\Gh\oplus
\left(
\bigoplus_{\alpha\in\Delta}\Gg_\alpha
\right),
\qquad
\Gg_\alpha=\{x\in\Gg\mid [h,x]=\alpha(h)x\;(h\in\Gh)\}.
\]
For $i\in I$ we take 
$\overline{e}_i\in\Gg_{\alpha_i}$, 
$\overline{f}_i\in\Gg_{-\alpha_i}$ such that
$[\overline{e}_i,\overline{f}_i]=\alpha_i^\vee$.

We denote by $B$, $B^+$ the Borel subgroups of $G$ with Lie algebras
\[
\Gb=\Gh\oplus
\left(
\bigoplus_{\alpha\in\Delta^+}\Gg_{-\alpha}
\right),
\qquad
\Gb^+=\Gh\oplus
\left(
\bigoplus_{\alpha\in\Delta^+}\Gg_{\alpha}
\right)
\]
respectively.
We denote by $N$, $N^+$ the unipotent radicals of $B$, $B^+$ respectively.
Their Lie algebras are given by
\[
\Gn=
\bigoplus_{\alpha\in\Delta^+}\Gg_{-\alpha},
\qquad
\Gn^+=
\bigoplus_{\alpha\in\Delta^+}\Gg_{\alpha}
\]
respectively.

\subsection{}
For a Lie algebra $\Ga$ we denote its enveloping algebra by $\overline{U}(\Ga)$.
For $\lambda\in X$ we define a character 
$\overline{\chi}_\lambda:\overline{U}(\Gh)\to \BC$ 
by
\[
\overline{\chi}_\lambda(h)=\langle \lambda,h\rangle
\qquad(h\in \Gh).
\]
For $n\in\BZ_{\geqq0}$ set
\[
\begin{pmatrix}
x
\\
n
\end{pmatrix}
=
\frac{x(x-1)\dots(x-n+1)}{n!}
\in\BQ[x].
\]
We denote by $\overline{U}_\BZ(\Gh)$  the $\BZ$-subalgebra of $\overline{U}(\Gh)$ generated by the elements
\[
\begin{pmatrix}
y
\\
n
\end{pmatrix}
\quad(y\in Y,\; n\in\BZ_{\geqq0}).
\]
For $i\in I$ and $n\in\BZ_{\geqq0}$ we define 
$\overline{e}_i^{(n)},\;\overline{f}_i^{(n)}\in \overline{U}(\Gg)$
by
\[
\overline{e}_i^{(n)}=\frac{\overline{e}_i^n}{n!},
\qquad
\overline{f}_i^{(n)}=\frac{\overline{f}_i^n}{n!}.
\]
We define $\BZ$-subalgebras 
$\overline{U}_\BZ(\Gn)$, $\overline{U}_\BZ(\Gn^+)$, 
$\overline{U}_\BZ(\Gb)$, 
$\overline{U}_\BZ(\Gg)$
of $\overline{U}(\Gg)$ by
\begin{gather*}
\overline{U}_\BZ(\Gn)
=
\langle 
\overline{f}^{(n)}_i\mid i\in I, n\geqq0
\rangle,
\qquad
\overline{U}_\BZ(\Gn^+)
=
\langle 
\overline{e}^{(n)}_i\mid i\in I, n\geqq0
\rangle,
\\
\overline{U}_\BZ(\Gb)
=
\langle 
\overline{U}_\BZ(\Gh), \overline{U}_\BZ(\Gn)
\rangle,
\qquad
\overline{U}_\BZ(\Gg)
=
\langle 
\overline{U}_\BZ(\Gh), \overline{U}_\BZ(\Gn), \overline{U}_\BZ(\Gn^+)
\rangle.
\end{gather*}

For a commutative ring $R$ we set
\begin{gather*}
\overline{U}_R(\Gh)=R\otimes_\BZ \overline{U}_\BZ(\Gh),
\quad
\overline{U}_R(\Gn)=R\otimes_\BZ \overline{U}_\BZ(\Gn),
\quad
\overline{U}_R(\Gn^+)=R\otimes_\BZ \overline{U}_\BZ(\Gn^+),
\\
\overline{U}_R(\Gb)=R\otimes_\BZ \overline{U}_\BZ(\Gb),
\quad
\overline{U}_R(\Gg)=R\otimes_\BZ \overline{U}_\BZ(\Gg).
\end{gather*}
They are  Hopf algebras over $R$.
Note that for $\lambda\in X$ we have $\overline{\chi}_\lambda(\overline{U}_\BZ(\Gh))\subset\BZ$.
Hence 
the character $\overline{\chi}_\lambda:\overline{U}(\Gh)\to \BC$ induces the character 
$\overline{\chi}_\lambda:\overline{U}_R(\Gh)\to R$ 
of $\overline{U}_R(\Gh)$.
\subsection{}
For an integer $m$ we define its $t$-analogues $[m]_t$ and 
$\{m\}_t$ by
\[
[m]_t=\frac{t^m-t^{-m}}{t-t^{-1}}\in\BZ[t, t^{-1}], \qquad
\{m\}_t=
\frac{t^m-1}{t-1}\in\BZ[t]
.
\]
For a non-negative integer $n$ we set
\[
[n]_t!=[1]_t[2]_t\dots[n]_t\in\BZ[t, t^{-1}], 
\qquad
\{n\}_t!=\{1\}_t\{2\}_t\dots\{n\}_t\in\BZ[t]
.
\]
We have
\[
[m]_t=t^{-m+1}\{m\}_{t^2},
\qquad
[n]_t!=t^{-n(n-1)/2}\{n\}_{t^2}!.
\]

\subsection{}
We fix a $W$-invariant symmetric bilinear form
\begin{equation}
\label{eq:bilin}
(\;,\;):
\sum_{\alpha\in\Delta}\BQ\alpha
\times 
\sum_{\alpha\in\Delta}\BQ\alpha
\to\BQ
\end{equation}
satisfying $(\alpha,\alpha)\in2\BZ$ for any $\alpha\in\Delta$.
For $\alpha\in\Delta$ we set $d_\alpha=(\alpha,\alpha)/2$, and for $i\in I$ we set $d_i=d_{\alpha_i}$.

Set $\BF=\BQ(q)$.
The quantized enveloping algebra $U_\BF(\Gg)$ is the associative algebra over $\BF$ with 1 generated by the elements
\[
k_y\;\;(y\in Y),\qquad
e_i,\;\; f_i\;\;(i\in I)
\]
satisfying the relations
\begin{align*}
&k_0=1,\qquad
k_{y_1}k_{y_2}=k_{y_1+y_2}\quad(y_1, y_2\in Y),
\\
&k_ye_i=q^{\langle\alpha_i,y\rangle}e_ik_y,
\quad
k_yf_i=q^{-\langle\alpha_i,y\rangle}f_ik_y
\quad(y\in Y, i\in I),
\\
&
e_if_j-f_je_i=\delta_{ij}\frac{k_i-k_i^{-1}}{q_i-q_i^{-1}}
\qquad(i, j\in I),
\\
&
\sum_{r=0}^{1-a_{ij}}
(-1)^r
e_i^{(1-a_{ij}-r)}e_je_i^{(r)}=
0
\qquad
(i, j\in I, \; i\ne j),
\\
&
\sum_{r=0}^{1-a_{ij}}
(-1)^r
f_i^{(1-a_{ij}-r)}f_jf_i^{(r)}=
0
\qquad
(i, j\in I, \; i\ne j).
\end{align*}
Here, $q_i=q^{d_i}$, $k_i=k_{d_i\alpha_i^\vee}$ for $i\in I$, 
and
$e_i^{(n)}=e_i^n/[n]_{q_i}!$, 
$f_i^{(n)}=f_i^n/[n]_{q_i}!$ for $i\in I$, $n\in\BZ_{\geqq0}$.

We will use the Hopf algebra structure of $U_\BF(\Gg)$ given by 
\begin{align*}
&
\Delta(k_y)=k_y\otimes k_y,
\\
&
\Delta(e_i)=e_i\otimes 1+k_i\otimes e_i,
\quad
\Delta(f_i)=f_i\otimes k_i^{-1}+1\otimes f_i,
\\
&
\varepsilon(k_y)=1,\qquad
\varepsilon(e_i)=\varepsilon(f_i)=0,
\\
&
S(k_y)=k_y^{-1},
\quad
S(e_i)=-k_i^{-1}e_i,
\quad
S(f_i)=-f_ik_i
\end{align*}
for $y\in Y$, $i\in I$.
We define $\BF$-subalgebras $U_\BF(\Gh)$, $U_\BF(\Gb)$, $U_\BF(\Gn)$, $U_\BF(\Gn^+)$ of $U_\BF(\Gg)$ by
\begin{align*}
&U_\BF(\Gh)=\langle k_y\mid y\in Y\rangle,
\qquad
U_\BF(\Gb)=\langle k_y, f_i\mid y\in Y, i\in I\rangle,
\\
&
U_\BF(\Gn)=\langle f_i\mid i\in I\rangle,
\qquad
U_\BF(\Gn^+)=\langle e_i\mid i\in I\rangle.
\end{align*}
Then we have
\[
U_\BF(\Gh)=\bigoplus_{y\in Y}\BF k_y.
\]
For $\lambda\in X$ we define a character 
$\chi_\lambda:U_\BF(\Gh)\to\BF$ by
\[
\chi_\lambda(k_y)=q^{\langle \lambda,y\rangle}
\qquad(y\in Y).
\]
For $\gamma\in Q^+$ set
\begin{align*}
U_\BF(\Gn)_{-\gamma}=&
\{
u\in U_\BF(\Gn)\mid k_yuk_y^{-1}=q^{-\langle\gamma,y\rangle}u\;\;(y\in Y)\},
\\
U_\BF(\Gn^+)_{\gamma}=&
\{
u\in U_\BF(\Gn^+)\mid k_yuk_y^{-1}=q^{\langle\gamma,y\rangle}u\;\;(y\in Y)\}.
\end{align*}
Then we have
\[
U_\BF(\Gn)=\bigoplus_{\gamma\in Q^+} U_\BF(\Gn)_{-\gamma},
\qquad
U_\BF(\Gn^+)=\bigoplus_{\gamma\in Q^+}
U_\BF(\Gn^+)_{\gamma}.
\]

\subsection{}
Set $\BA=\BZ[q,q^{-1}]$.
We define an $\BA$-subalgebra $U_\BA(\Gh)$ of $U_\BF(\Gh)$ by
\[
U_\BA(\Gh)=
\{u\in U_\BF(\Gh)
\mid
\chi_\lambda(u)\in\BA\;(\forall \lambda\in X)\}.
\]
By the definition of $U_\BA(\Gh)$, the character 
$\chi_\lambda:U_\BF(\Gh)\to\BF$ for $\lambda\in X$ induces an algebra homomorphism
$\chi_\lambda:U_\BA(\Gh)\to\BA$.
For 
$y\in Y$, $n\in\BZ_{\geqq0}$, $m\in\BZ$ we have
\[
k_y\in U_\BA(\Gh),
\qquad
\begin{Bmatrix}
q^mk_y
\\
n
\end{Bmatrix}_q
\in
U_\BA(\Gh),
\]
where, for $n\in\BZ_{\geqq0}$, we set
\[
\begin{Bmatrix}
x
\\
n
\end{Bmatrix}_t
=
\prod_{s=1}^n
\frac{xt^{-s+1}-1}{t^s-1}\in(\BQ(t))[x].
\]

The proof of the following result is easily reduced to the case where $Y$ is of rank one.
Details are omitted.
\begin{lemma}
\label{lem:ha}
Let $y_1,\dots, y_m$ be a basis of the free $\BZ$-module $Y$.
\begin{itemize}
\item[{\rm(i)}]{\rm(see \cite[Theorem 3.1]{DL})}
$U_\BA(\Gh)$ is a free $\BA$-module with basis
\[
\prod_{a=1}^m
\begin{Bmatrix}
k_{y_a}
\\
n_a
\end{Bmatrix}_q
k_{y_a}^{-\lfloor (n_a+1)/2\rfloor}
\qquad
(n_1,\dots, n_m\in\BZ_{\geqq0}).
\]
\item[{\rm(ii)}]
$U_\BA(\Gh)\cap
\BF[k_{y_1},\dots, k_{y_m}]$
 is a free $\BA$-module with basis
\[
\prod_{a=1}^m
\begin{Bmatrix}
k_{y_a}
\\
n_a
\end{Bmatrix}_q
\qquad
(n_1,\dots, n_m\in\BZ_{\geqq0}).
\]
\item[{\rm(iii)}]
The ring $U_\BA(\Gh)$ is the localization of $U_\BA(\Gh)\cap
\BF[k_{y_1},\dots, k_{y_m}]$ with respect to the multiplicative set 
$\{k_{n_1y_1+\dots+n_my_m}\mid
n_1,\dots, n_m\in\BZ_{\geqq0}\}$.
\end{itemize}
\end{lemma}

We denote by $U_\BA(\Gg)$ the $\BA$-subalgebra of $U_\BF(\Gg)$ generated by $U_\BA(\Gh)$ and $e_i^{(n)}$, $f_i^{(n)}$ for $i\in I$, $n\in\BZ_{\geqq0}$.
It is naturally a Hopf algebra over $\BA$.
We define $\BA$-subalgebras 
$U_\BA(\Gb)$, $U_\BA(\Gn)$, $U_\BA(\Gn^+)$
of $U_\BA(\Gg)$ by
\begin{gather*}
U_\BA(\Gb)=
\langle U_\BA(\Gh), 
f_i^{(n)}
\mid
i\in I, n\in\BZ_{\geqq0}\rangle,
\\
U_\BA(\Gn)=
\langle
f_i^{(n)}
\mid
i\in I, n\in\BZ_{\geqq0}\rangle,
\quad
U_\BA(\Gn^+)=
\langle
e_i^{(n)}
\mid
i\in I, n\in\BZ_{\geqq0}\rangle.
\end{gather*}
We have the triangular decomposition
\[
U_\BA(\Gg)
\cong U_\BA(\Gn^+)\otimes U_\BA(\Gh)\otimes
U_\BA(\Gn),
\qquad
U_\BA(\Gb)
\cong U_\BF(\Gh)\otimes
U_\BA(\Gn),
\]
where the isomorphisms are induced by the multiplication.
For $\gamma\in Q^+$ set
\[
U_\BA(\Gn)_{-\gamma}
=U_\BA(\Gn)\cap
U_\BF(\Gn)_{-\gamma},
\qquad
U_\BA(\Gn^+)_{\gamma}
=U_\BA(\Gn^+)\cap
U_\BF(\Gn)_{\gamma}.
\]
Then we have
\[
U_\BA(\Gn)=
\bigoplus_{\gamma\in Q^+}
U_\BA(\Gn)_{-\gamma},
\qquad
U_\BA(\Gn^+)=
\bigoplus_{\gamma\in Q^+}
U_\BA(\Gn^+)_{\gamma}.
\]
It is known that $U_\BA(\Gn)_{-\gamma}$ and $U_\BA(\Gn^+)_\gamma$ are free $\BA$-modules of finite rank (see \cite{Lroots}).
\subsection{}\label{subsec:UC}
Let $R$ be a commutative ring equipped with 
 $\zeta\in R^\times$.
We set
\[
U_{R,\zeta}(\Gg)=R\otimes_\BA U_\BA(\Gg),
\quad
U_{R,\zeta}(\Gb)=R\otimes_\BA U_\BA(\Gb),
\quad
U_{R,\zeta}(\Gh)=R\otimes_\BA U_\BA(\Gh),
\]
\[
U_{R,\zeta}(\Gn)=R\otimes_\BA U_\BA(\Gn),
\quad
U_{R,\zeta}(\Gn^+)=R\otimes_\BA U_\BA(\Gn^+),
\]
where $\BA\to R$ is given by $q\mapsto\zeta$.
Then $U_{R,\zeta}(\Gg)$ is a Hopf algebra over $R$, and 
$U_{R,\zeta}(\Gb)$, 
$U_{R,\zeta}(\Gh)$, 
$U_{R,\zeta}(\Gn)$,
$U_{R,\zeta}(\Gn^+)$ are naturally identified with $R$-subalgebras of $U_{R,\zeta}(\Gg)$.
Moreover, $U_{R,\zeta}(\Gb)$, 
$U_{R,\zeta}(\Gh)$ are Hopf subalgebras.
We have the triangular decomposition
\[
U_{R,\zeta}(\Gg)
\cong U_{R,\zeta}(\Gn)\otimes U_{R,\zeta}(\Gh)\otimes
U_{R,\zeta}(\Gn^+),
\qquad
U_{R,\zeta}(\Gb)
\cong U_{R,\zeta}(\Gn)\otimes
U_{R,\zeta}(\Gh).
\]
For $\gamma\in Q^+$ we set
\[
U_{R,\zeta}(\Gn)_{-\gamma}
=R\otimes_\BA U_\BA(\Gn)_{-\gamma},
\qquad
U_{R,\zeta}(\Gn^+)_{\gamma}
=R\otimes_\BA U_\BA(\Gn^+)_{\gamma}.
\]
Then we have
\[
U_{R,\zeta}(\Gn)=
\bigoplus_{\gamma\in Q^+}
U_{R,\zeta}(\Gn)_{-\gamma},
\qquad
U_{R,\zeta}(\Gn^+)=
\bigoplus_{\gamma\in Q^+}
U_{R,\zeta}(\Gn^+)_{\gamma}.
\]

\begin{lemma}
\label{lem:hr}
Let $y_1,\dots, y_m$ be a basis of the free $\BZ$-module $Y$.
\begin{itemize}
\item[{\rm(i)}]
$U_{R,\zeta}(\Gh)$ is a free $R$-module with basis
\[
1\otimes
\prod_{a=1}^m
\begin{Bmatrix}
k_{y_a}
\\
n_a
\end{Bmatrix}_q
k_{y_a}^{-\lfloor(n_a+1)/2\rfloor}
\qquad
(n_1,\dots, n_m\in\BZ_{\geqq0}).
\]
\item[{\rm(ii)}]
The elements 
\[
1\otimes
\prod_{a=1}^m
\begin{Bmatrix}
k_{y_a}
\\
n_a
\end{Bmatrix}_q
\qquad
(n_1,\dots, n_m\in\BZ_{\geqq0})
\]
of $U_{R,\zeta}(\Gh)$ are linearly independent over $R$.
\item[{\rm(iii)}]
Denote by $U^+_{R,\zeta}(\Gh)$ the $R$-submodule of $U_{R,\zeta}(\Gh)$ generated by the elements in {\rm(ii)}.
Then $U^+_{R,\zeta}(\Gh)$ is a subring of $U_{R,\zeta}(\Gh)$.
Moreover,
$U_{R,\zeta}(\Gh)$ is a localization of $U^+_{R,\zeta}(\Gh)$ with respect to the multiplicative set 
$\{k_{n_1y_1+\dots+n_my_m}\mid
n_1,\dots, n_m\in\BZ_{\geqq0}\}$.
Hence
\begin{align*}
U_{R,\zeta}(\Gh)
=&
\bigcup_{n_1,\dots, n_m\in\BZ_{\geqq0}}
k_{n_1y_1+\dots+n_my_m}^{-1}
U^+_{R,\zeta}(\Gh).
\end{align*}
\end{itemize}
\end{lemma}
\begin{proof}
Set 
\[
L=R\otimes_\BA(U_\BA(\Gh)\cap\BF[k_{y_1},\dots, k_{y_m}]),
\quad
S=\{k_{n_1y_1+\dots+n_my_m}\mid
n_1,\dots, n_m\in\BZ_{\geqq0}\}\subset L.
\]
In view of Lemma \ref{lem:ha} it is sufficient to verify that the canonical homomorphism $L\to S^{-1}L$ is injective.
Hence we have only to show that the map
$L\ni z\mapsto k_{y_a}z\in L$ is injective for any $a$.
This is easily reduced to the case where $Y$ is of rank one.
Details are omitted.
\end{proof}

Let $\lambda\in X$.
By abuse of notation we denote by
\[
\chi_\lambda:U_{R,\zeta}(\Gh)\to R
\]
the $R$-algebra homomorphism induced by 
$\chi_\lambda:U_\BF(\Gh)\to\BF$.

The proof of the following fact is reduced to the rank one case.
Details are omitted.
\begin{lemma}
\label{lem:cr}
\begin{itemize}
\item[{\rm(i)}]
The subset $\{\chi_\lambda\mid \lambda\in X\}$ 
of $\Hom_R(U_{R,\zeta}(\Gh),R)$ 
is linearly independent over $R$.
\item[{\rm(ii)}]
Let $h\in U_{R,\zeta}(\Gh)$.
If $\chi_\lambda(h)=0$ for any $\lambda\in X$, then we have $h=0$.
\end{itemize}
\end{lemma}

We set
\[
\zeta_\alpha=\zeta^{d_\alpha}\;\;(\alpha\in\Delta),
\qquad
\zeta_i=\zeta^{d_i}\;\;(i\in I).
\]
\subsection{}
In this subsection we assume that 
$\zeta_\alpha=\pm1$ for any $\alpha\in\Delta$.
We compare $U_{R,\zeta}(\Gg)$ with $\overline{U}_R(\Gg)$ in the following.

Note that $\zeta^s=1$ for some $s\in\BZ_{>0}$ by our assumption.
For $y\in Y$ and $n\in\BZ_{\geqq0}$ 
we set
\[
h(y,n)
=
1\otimes
\begin{Bmatrix}
k_{sy}
\\
n
\end{Bmatrix}_{q^s}
\in U_{R,\zeta}(\Gh).
\]
By Lemma \ref{lem:cr} (ii) it is characterized  as the element of $U_{R,\zeta}(\Gh)$ satisfying
\[
\chi_\lambda(h(y,n))=
\begin{pmatrix}
\langle\lambda,y\rangle
\\
n
\end{pmatrix}1_R
\]
for any $\lambda\in X$.
In particular, $h(y,n)$ does not depend on the choice of $s$.
Denote by $U'_{R,\zeta}(\Gh)$ the subalgebra of $U'_{R,\zeta}(\Gh)$ generated by the elements $h(y,n)$ for $y\in Y$, $n\in\BZ_{\geqq0}$.
By Lemma \ref{lem:cr} (ii) we see easily the following
\begin{lemma}
If $y_1,\dots, y_m$ is a basis of the free $\BZ$-module $Y$, then 
\[
\prod_{a=1}^m h(y_a,n_a)
\qquad
(n_1,\dots, n_m\in\BZ_{\geqq0})
\]
form a basis of the $R$-module $U'_{R,\zeta}(\Gh)$.
\end{lemma}

For $i\in I$, $n\in\BZ_{\geqq0}$, $s\in\BZ$ set
\begin{align*}
\begin{bmatrix}
k_{i};s
\\
n
\end{bmatrix}_{q_i}
=&
\prod_{a=1}^n
\frac
{q_i^{s-a+1}k_i-q_i^{-s+a-1}k_i^{-1}}
{q_i^a-q_i^{-a}}
\in U_{\BA}(\Gh),
\\
t(i,n,s)
=&1\otimes q_i^{-n(n-s)}k_i^n\begin{bmatrix}
k_{i};s
\\
n
\end{bmatrix}_{q_i}
\in U_{R,\zeta}(\Gh).
\end{align*}
Then we have
\[
\chi_\lambda(t(i,n,s))=
\begin{pmatrix}
\langle\lambda,\alpha_i^\vee\rangle+s
\\
n
\end{pmatrix}
1_R
\]
for any $\lambda\in X$.
Note 
\[
\begin{pmatrix}
x+s
\\
n
\end{pmatrix}
\in\sum_{m=0}^n
\BZ
\begin{pmatrix}
x
\\
m
\end{pmatrix}
\]
in $\BQ[x]$.
Hence by Lemma \ref{lem:cr} (ii) we have
\[
t(i,n, s)\in
\sum_{m=0}^n
R h(\alpha_i^\vee,m)\in U'_{R,\zeta}(\Gh)
\qquad
(i\in I, n\in\BZ_{\geqq0}).
\]

Take a subset $J$ of $I$ satisfying
\begin{equation}
\label{eq:J}
i, j\in I, \;
a_{ij}<0
\Longrightarrow
\;
|J\cap\{i, j\}|=1.
\end{equation}
For $i\in I$ and $n\in\BZ_{\geqq0}$ we define 
$e(i,n), f(i,n)
\in
U_{R,\zeta}(\Gg)$
by
\[
e(i,n)
=
\begin{cases}
\zeta_i^{n(n-1)/2}e_i^{(n)}&(i\in J)
\\
\zeta_i^{n(n+1)/2}e_i^{(n)}k_i^n&(i\notin J),
\end{cases}
\qquad
f(i,n)=
\begin{cases}
\zeta_i^{n(n+1)/2}f_i^{(n)}k_i^n&(i\in J)
\\
\zeta_i^{n(n-1)/2}f_i^{(n)}&(i\notin J).
\end{cases}
\]
We denote by $U'_{R,\zeta}(\Gn)$ (resp.\ $U'_{R,\zeta}(\Gn^+)$)
the $R$-subalgebra of $U_{R,\zeta}(\Gg)$ generated by
$\{
f(i,n)\mid i\in J, n\in\BZ_{\geqq0}
\}$
(resp.\
$\{
e(i,n)\mid i\in J, n\in\BZ_{\geqq0}
\}$).
For $\gamma=\sum_{i\in I}m_i\alpha_i\in Q$ we set
$\gamma^\dagger=\sum_{i\in J}m_i\alpha_i\in Q$.
Then we have
\begin{equation}
\label{eq:UU-prime}
U'_{R,\zeta}(\Gn)=
\bigoplus_{\gamma\in Q^+}U_{R,\zeta}(\Gn)_{-\gamma}k_{\gamma^\dagger},
\qquad
U'_{R,\zeta}(\Gn^+)=
\bigoplus_{\gamma\in Q^+}U_{R,\zeta}(\Gn^+)_{\gamma}k_{\gamma-\gamma^\dagger}.
\end{equation}

We define $U'_{R,\zeta}(\Gg)$ to be the $R$-subalgebra of $U_{R,\zeta}(\Gg)$ generated by 
$U'_{R,\zeta}(\Gh)$, $U'_{R,\zeta}(\Gn)$, $U'_{R,\zeta}(\Gn^+)$.

By well-known relations in $U_\BF(\Gg)$ we have
\begin{align}
\label{eq:cc1}
&he(i,n)=
\chi_{n\alpha_i}(h)
e(i,n)h,
\qquad
hf(i,n)=
\chi_{-n\alpha_i}(h)
f(i,n)h,
\\
\label{eq:cc2}
&
e(i,n)f(j,m)=f(j,m)e(i,n)\qquad(i\ne j),
\\
\label{eq:cc3}
&e(i,n)f(i,m)
=
\sum_{0\leqq a\leqq n,m}
f(i,m-a)t(i,a,2a-m-n)e(i,n-a)
\end{align}
for $h\in U'_{R,\zeta}(\Gh)$, $n, m\in\BZ_{\geqq0}$, $i, j\in I$.
By \eqref{eq:cc1}, \eqref{eq:cc2}, \eqref{eq:cc3} we see that 
the multiplication of $U'_{R,\zeta}(\Gg)$ induces the isomorphism
\begin{equation}
\label{eq:triUprime}
U'_{R,\zeta}(\Gg)
\cong
U'_{R,\zeta}(\Gn)
\otimes
U'_{R,\zeta}(\Gh)
\otimes
U'_{R,\zeta}(\Gn^+)
\end{equation}
of $R$-modules.
\begin{lemma}
\label{lem:u-prime}
Recall that $\zeta_\alpha=\pm1$ for any $\alpha\in \Delta$.
We have isomorphisms
\begin{align*}
\overline{U}_R(\Gn)&\cong U'_{R,\zeta}(\Gn)
\qquad(\overline{f}_i^{(n)}\leftrightarrow f(i,n)),
\\
\overline{U}_R(\Gn^+)&\cong U'_{R,\zeta}(\Gn^+)
\qquad(\overline{e}_i^{(n)}\leftrightarrow e(i,n))
\end{align*}
of $R$-algebras.
\end{lemma}
\begin{proof}
We may assume $\Gg$ is simple.
By considering the situation where the bilinear form $(\;,\;)$ in \eqref{eq:bilin} is replaced by 
$d_\alpha^{-1}(\;,\;)$ for short roots $\alpha$, 
we may further assume that $q=q_\alpha$ for short roots $\alpha$.
In this case $\zeta=\pm1$, and we have 
\[
U'_{R,\zeta}(\Gn)=R\otimes_\BZ U'_{\BZ,\xi}(\Gn),
\qquad
U'_{R,\zeta}(\Gn^+)=R\otimes_\BZ U'_{\BZ,\xi}(\Gn^+).
\]
Here $\xi=1\in\BZ$ if $\zeta=1\in R$, and $\xi=-1\in\BZ$ if $\zeta=-1\in R$.
Hence we have only to show
\[
\overline{U}_\BZ(\Gn)\cong U'_{\BZ,\xi}(\Gn),
\qquad
\overline{U}_\BZ(\Gn^+)\cong U'_{\BZ,\xi}(\Gn^+).
\]
By using the embeddings
\[
\overline{U}_\BZ(\Gn)\subset \overline{U}_\BQ(\Gn),
\;\;
\overline{U}_\BZ(\Gn^+)\subset \overline{U}_\BQ(\Gn^+),
\;\;
U'_{\BZ,\xi}(\Gn)
\subset
U'_{\BQ,\xi}(\Gn),
\;\;
U'_{\BZ,\xi}(\Gn^+)
\subset
U'_{\BQ,\xi}(\Gn^+)
\]
the proof is reduced to showing 
\[
\overline{U}_\BQ(\Gn)\cong U'_{\BQ,\xi}(\Gn),
\qquad
\overline{U}_\BQ(\Gn^+)\cong U'_{\BQ,\xi}(\Gn^+).
\]
This is easily verified by checking the Serre type relations.
\end{proof}
In view of the relations 
\eqref{eq:cc1}, \eqref{eq:cc2}, \eqref{eq:cc3} we see from 
\eqref{eq:triUprime}, Lemma \ref{lem:u-prime} the following.
\begin{proposition}
\label{prop:pm1}
Recall that $\zeta_\alpha=\pm1$ for any $\alpha\in \Delta$.
We have an isomorphism 
\begin{equation}
\label{eq:pm1}
\overline{U}_R(\Gg)\cong U'_{R,\zeta}(\Gg)
\end{equation}
of $R$-algebras given by
\begin{align}
\label{eq:isomU}
\begin{pmatrix}
y\\n
\end{pmatrix}
\leftrightarrow
h(y,n),
\qquad
\overline{e}_i^{(n)}
\leftrightarrow
e(i,n),
\qquad
\overline{f}_i^{(n)}
\leftrightarrow
f(i,n)
\end{align}
for $y\in Y$, $i\in I$, $n\in\BZ_{\geqq0}$ $($compare \cite[Proposition 33.2.3]{Lbook}$)$.
\end{proposition}

\begin{remark}
\label{rem:isomU}
{\rm
For $i\in I$ satisfying $\zeta_i=1$ we have $k_i=1$ in $U_{R,\zeta}(\Gh)$ by Lemma \ref{lem:cr}.
Hence if $\zeta_i=1$ for any $i\in I$, then \eqref{eq:pm1}
turns out to be a Hopf algebra isomorphism.
In general \eqref{eq:pm1}
does not 
preserve the comultiplication.
}
\end{remark}

\subsection{}
Let $R$ be a commutative ring equipped with $\zeta\in R^\times$.
In this subsection we assume that 
there exists some $\ell\in\BZ_{>0}$ such that $f_\ell(\zeta)=0$, where $f_\ell$ is the $\ell$-th cyclotomic polynomial.
As in \cite{Lbook} we define a root datum $({}^\sharp X,{}^\sharp\Delta,{}^\sharp Y,{}^\sharp\Delta^\vee)$ as follows.
Let
\[
r=
\begin{cases}
\ell\quad&(\ell\in 2\BZ+1)
\\
\frac{\ell}2&
(\ell\in 2\BZ),
\end{cases}
\qquad
r_\alpha
=
\frac{r}{(r,d_\alpha)}\quad(\alpha\in\Delta),
\qquad
r_i
=
r_{\alpha_i}\quad(i\in I).
\]
For $\alpha\in\Delta$ we set
\[
{}^\sharp\alpha=r_\alpha\alpha\in X,
\qquad
{}^\sharp\alpha^\vee=r_\alpha^{-1}\alpha^\vee\in Q\otimes_\BZ Y.
\]
Then ${}^\sharp\Delta=\{{}^\sharp\alpha\mid\alpha\in \Delta\}$ turns out to be a root system with $\{{}^\sharp\alpha_i\mid i\in I\}$ a set of simple roots.
Moreover,
${}^\sharp\Delta^\vee=\{{}^\sharp\alpha^\vee\mid\alpha\in \Delta\}$
is the set of coroots for the root system ${}^\sharp\Delta$.
Set
\[
{}^\sharp X=\{\lambda\in X\mid
\langle\lambda,{}^\sharp\Delta^\vee\rangle\subset \BZ
\},
\qquad
{}^\sharp Y=
\{y\in\BQ\otimes_\BZ Y\mid
\langle {}^\sharp X,y\rangle\subset\BZ\}.
\]
Then  $({}^\sharp X,{}^\sharp\Delta,{}^\sharp Y,{}^\sharp\Delta^\vee)$ is a root datum.
The Weyl group of $({}^\sharp X,{}^\sharp\Delta,{}^\sharp Y,{}^\sharp\Delta^\vee)$ is naturally identified with the Weyl group $W$ of $(X,\Delta,Y,\Delta^\vee)$.
We set
\[
{}^\sharp X^+=\{\lambda\in {}^\sharp X\mid
\langle\lambda,{}^\sharp\alpha_i^\vee\rangle\geqq0\;\;(i\in I)
\}
={}^\sharp X\cap X^+.
\]

Let ${}^\sharp G$ be the connected reductive algebraic group over $\BC$ with root datum $({}^\sharp X,{}^\sharp\Delta,{}^\sharp Y,{}^\sharp\Delta^\vee)$, and let ${}^\sharp \Gg$ be its Lie algebra.
We denote by $U_{R,\zeta}({}^\sharp\Gg)$ the quantized enveloping algebra over $R$ associated to the root datum 
$({}^\sharp X,{}^\sharp\Delta,{}^\sharp Y,{}^\sharp\Delta^\vee)$
and the $W$-invariant symmetric bilinear form \eqref{eq:bilin} on
$
\sum_{\alpha\in\Delta}\BQ{}^\sharp\alpha
=
\sum_{\alpha\in\Delta}\BQ\alpha
$.
We similarly define $R$-subalgebras 
$U_{R,\zeta}({}^\sharp\Gh)$, 
$U_{R,\zeta}({}^\sharp\Gn)$, 
$U_{R,\zeta}({}^\sharp\Gn^+)$ 
of
$U_{R,\zeta}({}^\sharp\Gg)$.

Note that for any $\alpha\in\Delta$ we have $\zeta_{{}^\sharp\alpha}=\zeta^{r_\alpha^2}=\pm1$.
Hence we can apply the results in the preceding subsection to $U_{R,\zeta}({}^\sharp\Gg)$.
In particular, we have
$\overline{U}_R({}^\sharp\Gg)\cong U'_{R,\zeta}({}^\sharp\Gg)\subset U_{R,\zeta}({}^\sharp\Gg)$.

Following Lusztig we define the quantum Frobenius homomorphism
\begin{equation}
\CF:U_{R,\zeta}(\Gg)\to U_{R,\zeta}({}^\sharp\Gg)
\end{equation}
as follows.
By $Y\subset {}^\sharp Y$ we have an inclusion
$
U_\BF(\Gh)\subset U_\BF({}^\sharp \Gh)
$
sending $k_y$ for $y\in Y$ to $k_y$ for $y\in Y\subset {}^\sharp Y$.
This induces $U_\BA(\Gh)\subset U_\BA({}^\sharp\Gh)$ because ${}^\sharp X\subset X$.
Hence we obtain a natural homomorphism
\[
\CF_\Gh:U_{R,\zeta}(\Gh)\to U_{R,\zeta}({}^\sharp\Gh).
\]
On the other hand by \cite[Theorem 35.1.7]{Lbook} we have well-defined algebra homomorphisms
\[
\CF_\Gn:U_{R,\zeta}(\Gn)\to U_{R,\zeta}({}^\sharp\Gn),
\qquad
\CF_{\Gn^+}:U_{R,\zeta}(\Gn^+)\to U_{R,\zeta}({}^\sharp\Gn^+)
\]
satisfying
\[
\CF_\Gn(f_i^{(n)})
=\begin{cases}
f_i^{(n/r_i)}&(r_i|n)
\\
0&(\text{otherwise}),
\end{cases}
\quad
\CF_{\Gn^+}(e_i^{(n)})
=\begin{cases}
e_i^{(n/r_i)}&(r_i|n)
\\
0&(\text{otherwise})
\end{cases}
\]
for $i\in I$, $n\in\BZ_{\geqq0}$ 
(see also the last paragraph of  \cite[35.5.2]{Lbook}).
The following result is proved exactly as in \cite[Theorem 35.1.9]{Lbook} using Lemma \ref{lem:cr} (ii).
\begin{proposition}
There exists a unique Hopf algebra homomorphism
\[
\CF:U_{R,\zeta}(\Gg)\to U_{R,\zeta}({}^\sharp\Gg)
\]
satisfying 
$\CF|_{U_{R,\zeta}(\Gh)}=\CF_\Gh$, 
$\CF|_{U_{R,\zeta}(\Gn)}=\CF_\Gn$, 
$\CF|_{U_{R,\zeta}(\Gn^+)}=\CF_{\Gn^+}$.
\end{proposition}

\section{Representations}
\subsection{}
Let $\CH$ be a Hopf algebra over a commutative ring $R$.
For left $\CH$-modules $M_1$, $M_2$ we regard the left 
$\CH\otimes\CH$-module $M_1\otimes M_2=M_1\otimes_RM_2$ as a left $\CH$-module via the comultiplication $\Delta:\CH\to\CH\otimes\CH$.
For a left (resp.\ right) $\CH$-module $M$ we regard  $M^*=\Hom_R(M,R)$ as a right (resp.\ left) $\CH$-module by
\begin{align*}
\langle m^*h,m\rangle
=\langle m^*, hm\rangle,
\qquad
(\text{resp.}\;\;
\langle hm^*,m\rangle
=\langle m^*, mh\rangle)
\end{align*}
for $h\in \CH$, $m\in M$, $m^*\in M^*$.

\subsection{}
For a left (resp. right) $\overline{U}(\Gg)$-module $V$ and $\mu\in X$ we set 
\[
V_\mu=\{v\in V\mid hv=\mu(h)v\;(h\in\Gh)\},
\quad
(\text{resp.} \;V_\mu=\{v\in V\mid vh=\mu(h)v\;(h\in\Gh)\}).
\]
For $\lambda\in X^+$ we define a $\overline{U}(\Gg)$-module $\overline{V}(\lambda)$ by
\begin{align*}
\overline{V}(\lambda)
=&\overline{U}(\Gg)/
\left(
\sum_{h\in \Gh}
\overline{U}(\Gg)(h-\lambda(h))
+\sum_{i\in I}\overline{U}(\Gg)\overline{e}_i
+\sum_{i\in I}\overline{U}(\Gg)
\overline{f}_i^{\langle\lambda,\alpha_i^\vee\rangle+1}
\right).
\end{align*}
Then $\overline{V}(\lambda)$ is a finite-dimensional  irreducible $\overline{U}(\Gg)$-module, which has the weight space decomposition
$
\overline{V}(\lambda)
=
\bigoplus_{\mu\in X}\overline{V}(\lambda)_\mu
$.
Set $\overline{v}_\lambda=\overline{1}\in \overline{V}(\lambda)$, so that 
$\overline{V}(\lambda)_\lambda=\BC\overline{v}_\lambda$.
Set $\overline{V}^*(\lambda)=\Hom_\BC(\overline{V}(\lambda),\BC)$.
It is a finite-dimensional  irreducible right $\overline{U}(\Gg)$-module.
The weight space decomposition 
of $\overline{V}(\lambda)$ gives the 
weight space decomposition 
$
\overline{V}^*(\lambda)
=
\bigoplus_{\mu\in X}\overline{V}^*(\lambda)_\mu
$ of $\overline{V}^*(\lambda)$, where 
$\overline{V}^*(\lambda)_\mu=
(\overline{V}(\lambda)_\mu)^*$.
We define 
$\overline{v}^*_\lambda\in 
\overline{V}^*(\lambda)_\lambda$ by $\langle \overline{v}^*_\lambda,\overline{v}_\lambda\rangle=1$.

For $\lambda\in X^+$ we define a $\overline{U}_\BZ(\Gg)$-submodule $\overline{\Delta}_\BZ(\lambda)$ of $\overline{V}(\lambda)$ and 
a right $\overline{U}_\BZ(\Gg)$-submodule ${\overline{\Delta}}^*_\BZ(\lambda)$ of $\overline{V}^*(\lambda)$
by
\begin{align*}
\overline{\Delta}_\BZ(\lambda)=\overline{U}_\BZ(\Gg)\overline{v}_\lambda,
\qquad
{\overline{\Delta}}_\BZ^*(\lambda)=\overline{v}^*_\lambda \overline{U}_\BZ(\Gg)
.
\end{align*}
Then 
$\overline{\Delta}_\BZ(\lambda)$ and 
$\overline{\Delta}_\BZ^*(\lambda)$ 
are free $\BZ$-modules satisfying 
\[
\BC\otimes_\BZ \overline{\Delta}_\BZ(\lambda)\cong \overline{V}(\lambda), \quad
\BC\otimes_\BZ \overline{\Delta}^*_\BZ(\lambda)\cong 
\overline{V}^*(\lambda).
\]
For $\lambda\in X^+$ we  define a $\overline{U}_\BZ(\Gg)$-submodule 
$\overline{\nabla}_\BZ(\lambda)$ of $\overline{V}(\lambda)$ and 
a right $\overline{U}_\BZ(\Gg)$-submodules 
$\overline{\nabla}^*_\BZ(\lambda)$ of $\overline{V}^*(\lambda)$
by
\begin{align*}
\overline{\nabla}_\BZ(\lambda)=&
\{v\in \overline{V}(\lambda)
\mid
\langle\overline{\Delta}^*_\BZ(\lambda),v\rangle
\subset\BZ\}
\cong\Hom_\BZ(\overline{\Delta}^*_\BZ(\lambda),\BZ),
\\
\overline{\nabla}_\BZ^*(\lambda)=&
\{v^*\in \overline{V}^*(\lambda)
\mid
\langle v^*,\overline{\Delta}_\BZ(\lambda)\rangle
\subset\BZ\}
\cong\Hom_\BZ(\overline{\Delta}_\BZ(\lambda),\BZ).
\end{align*}
Then we have
\[
\overline{\Delta}_\BZ(\lambda)\subset \overline{\nabla}_\BZ(\lambda)
\subset \overline{V}(\lambda),
\qquad
\overline{\Delta}^*_\BZ(\lambda)\subset \overline{\nabla}^*_\BZ(\lambda)
\subset
\overline{V}^*(\lambda).
\]
Moreover,
$\overline{\nabla}_\BZ(\lambda)$ and 
$\overline{\nabla}_\BZ^*(\lambda)$ 
are free $\BZ$-modules satisfying 
\[
\BC\otimes_\BZ \overline{\nabla}_\BZ(\lambda)\cong \overline{V}(\lambda),
\quad
\BC\otimes_\BZ \overline{\nabla}^*_\BZ(\lambda)\cong 
\overline{V}^*(\lambda).
\]

Let $R$ be a commutative ring.
For a $\overline{U}_R(\Gg)$-module $V$ and $\mu\in X$ we set 
\[
V_\mu=\{v\in V\mid hv=\overline{\chi}_\mu(h)v\;(h\in\overline{U}_R(\Gh))\}.
\]
We say that 
a $\overline{U}_R(\Gg)$-module $V$ is integrable if it has the weight space decomposition
$V=\bigoplus_{\mu\in X}V_\mu$, and for any $v\in V$ and $i\in I$ we have 
$\overline{e}_i^{(n)}v=\overline{f}_i^{(n)}v=0$ 
for $n\gg0$.
We denote by $\Mod_{\inte}(\overline{U}_R(\Gg))$ the category of integrable $\overline{U}_R(\Gg)$-modules

Let $V$ be an integrable $\overline{U}_R(\Gg)$-module.
For $i\in I$ we define an invertible $R$-homomorphism 
$
\overline{T}_i|_V=\overline{T}_i:V\to V$
by 
\[
\overline{T}_iv=
\sum_{a-b+c=\langle\lambda,\alpha_i^\vee\rangle}
(-1)^bf_i^{(a)}e_i^{(b)}f_i^{(c)}v
\qquad(\lambda\in X, v\in V_\lambda).
\]
For $w\in W$  we define an invertible $R$-homomorphism 
$
\overline{T}_w|_V=\overline{T}_w:V\to V$
by 
$
 \overline{T}_w=
\overline{T}_{i_1}\dots
\overline{T}_{i_N}$,
where $w=s_{i_1}\dots s_{i_N}$ is a reduced expression of $w$.
It does not depend on the choice of a reduced expression.
Moreover, 
we have
$
\overline{T}_wV_\lambda=V_{w\lambda}
$ for $\lambda\in X$.
Regarding $\overline{U}_R(\Gg)$ as an integrable $\overline{U}_R(\Gg)$-module via the adjoint action we have
\begin{equation}
(\overline{T}_w|_V)(uv)=
(\overline{T}_w|_{\overline{U}_R(\Gg)}(u))(\overline{T}_w|_V(v))
\qquad
(u\in \overline{U}_R(\Gg), v\in V).
\end{equation}
The following is well-known.

\begin{lemma}
\label{lem:oTtens}
For integrable $\overline{U}_{R}(\Gg)$-modules $V_1$, $V_2$ 
we have
\[
\overline{T}_w|_{V_1\otimes V_2}=
\overline{T}_w|_{V_1}
\otimes
\overline{T}_w|_{V_2}.
\]
\end{lemma}

For $\lambda\in X^+$ we define left $\overline{U}_R(\Gg)$-modules $\overline{\Delta}_R(\lambda)$, $\overline{\nabla}_R(\lambda)$ and right $\overline{U}_R(\Gg)$-modules $\overline{\Delta}^*_R(\lambda)$ $\overline{\nabla}^*_R(\lambda)$ as the base changes of 
$\overline{\Delta}_\BZ(\lambda)$, $\overline{\nabla}_\BZ(\lambda)$,  $\overline{\Delta}^*_\BZ(\lambda)$ $\overline{\nabla}^*_\BZ(\lambda)$
respectively.
We call $\overline{\Delta}_R(\lambda)$, $\overline{\Delta}^*_R(\lambda)$ the (left and right) Weyl modules with highest weight $\lambda$, and
$\overline{\nabla}_R(\lambda)$, $\overline{\nabla}^*_R(\lambda)$
the (left and right) dual Weyl modules with highest weight $\lambda$.
The left $\overline{U}_R(\Gg)$-modules 
$\overline{\Delta}_R(\lambda)$, $\overline{\nabla}_R(\lambda)$ are integrable.

\subsection{}
For a left (resp. right) ${U}_\BF(\Gg)$-module $V$ and $\mu\in X$ we set 
\begin{gather*}
V_\mu=\{v\in V\mid hv=\chi_\mu(h)v\;(h\in{U}_\BF(\Gh))\},
\\
(\text{resp.} \;\;V_\mu=\{v\in V\mid vh=\chi_\mu(h)v\;(h\in{U}_\BF(\Gh))\}).
\end{gather*}
For $\lambda\in X^+$ we define a $U_\BF(\Gg)$-module $V_\BF(\lambda)$ by
\begin{align*}
V_\BF(\lambda)
=U_\BF(\Gg)/
\left(
\sum_{h\in U_\BF(\Gh)}U_\BF(\Gg)(h-\chi_\lambda(h))
+\sum_{i\in I}U_\BF(\Gg)e_i
+\sum_{i\in I}U_\BF(\Gg)
f_i^{\langle\lambda,\alpha_i^\vee\rangle+1}
\right).
\end{align*}
Then $V_\BF(\lambda)$ is a finite-dimensional  irreducible $U_\BF(\Gg)$-module.
Set $V^*_\BF(\lambda)=\Hom_\BF(V_\BF(\lambda),\BF)$.
It is a finite-dimensional  irreducible right $U_\BF(\Gg)$-module.
The weight space decomposition 
$
V_\BF(\lambda)=\bigoplus_{\mu\in X}V_\BF(\lambda)_\mu
$
of $V_\BF(\lambda)$ gives the weight space decomposition 
$
V^*_\BF(\lambda)=\bigoplus_{\mu\in X}V^*_\BF(\lambda)_\mu
$ of 
$V^*_\BF(\lambda)$, where 
$
V_\BF^*(\lambda)_\mu=
(V_\BF(\lambda)_\mu)^*
$.
Set $v_\lambda=\overline{1}\in V_\BF(\lambda)$.
Then we have $V_\BF(\lambda)_\lambda=\BF v_\lambda$.
We define $v^*_\lambda\in V^*_\BF(\lambda)_\lambda$ by $\langle v^*_\lambda,v_\lambda\rangle=1$.

\subsection{}
For $\lambda\in X^+$ we define a $U_\BA(\Gg)$-submodule $\Delta_\BA(\lambda)$ of $V_\BF(\lambda)$ and 
a right $U_\BA(\Gg)$-submodule $\Delta^*_\BA(\lambda)$ of $V^*_\BF(\lambda)$
by
\begin{align*}
\Delta_\BA(\lambda)=U_\BA(\Gg)v_\lambda,
\qquad
\Delta_\BA^*(\lambda)=v^*_\lambda U_\BA(\Gg)
.
\end{align*}
They have the weight space decomposition 
\[
\Delta_\BA(\lambda)
=
\bigoplus_{\mu\in X}\Delta_\BA(\lambda)_\mu,
\quad
\Delta^*_\BA(\lambda)
=
\bigoplus_{\mu\in X}\Delta^*_\BA(\lambda)_\mu,
\]
where
\begin{align*}
\Delta_\BA(\lambda)_\mu
=&\,\{v\in \Delta_\BA(\lambda)\mid hv=\chi_\mu(h)v\;(h\in U_\BA(\Gh))\},
\\
\Delta^*_\BA(\lambda)_\mu
=&\,\{v\in \Delta^*_\BA(\lambda)\mid vh=\chi_\mu(h)v\;(h\in U_\BA(\Gh))\}.
\end{align*}
It follows from  the deep theory of canonical bases that $\Delta_\BA(\lambda)_\mu$ and $\Delta^*_\BA(\lambda)_\mu$ are free $\BA$-modules (see \cite{Lbook}).
In particular, 
$\Delta_\BA(\lambda)$ and 
$\Delta_\BA^*(\lambda)$ 
are free $\BA$-modules satisfying 
\[
\BF\otimes_\BA \Delta_\BA(\lambda)\cong V_\BF(\lambda), \quad
\BF\otimes_\BA \Delta^*_\BA(\lambda)\cong V^*_\BF(\lambda).
\]
We define a $U_\BA(\Gg)$-submodule 
$\nabla_\BA(\lambda)$ of $V_\BF(\lambda)$ and 
a right $U_\BA(\Gg)$-submodules 
$\nabla^*_\BA(\lambda)$ of $V^*_\BF(\lambda)$
by
\begin{align*}
\nabla_\BA(\lambda)=&
\{v\in V_\BF(\lambda)\mid
\langle 
\Delta^*_\BA(\lambda),v\rangle\subset\BA\}
\cong
\Hom_\BA(\Delta^*_\BA(\lambda),\BA),
\\
\nabla_\BA^*(\lambda)=&
\{v^*\in V^*_\BF(\lambda)\mid
\langle 
v^*,\Delta_\BA(\lambda)\rangle\subset\BA\}
\cong
\Hom_\BA(\Delta_\BA(\lambda),\BA).
\end{align*}
Then we have
\[
\Delta_\BA(\lambda)\subset \nabla_\BA(\lambda)
\subset
V_\BF(\lambda),
\qquad
\Delta^*_\BA(\lambda)\subset \nabla^*_\BA(\lambda)
\subset
V^*_\BF(\lambda).
\]
We have the weight space decomposition 
\[
\nabla_\BA(\lambda)
=
\bigoplus_{\mu\in X}\nabla_\BA(\lambda)_\mu,
\quad
\nabla^*_\BA(\lambda)
=
\bigoplus_{\mu\in X}\nabla^*_\BA(\lambda)_\mu,
\]
where
\begin{align*}
\nabla_\BA(\lambda)_\mu
=&\{v\in \nabla_\BA(\lambda)\mid hv=\chi_\mu(h)v\;(h\in U_\BA(\Gh))\}
=\Hom_\BA(\Delta^*_\BA(\lambda)_\mu,\BA),
\\
\nabla^*_\BA(\lambda)_\mu
=&\{v\in \nabla^*_\BA(\lambda)\mid vh=\chi_\mu(h)v\;(h\in U_\BA(\Gh))\}
=\Hom_\BA(\Delta_\BA(\lambda)_\mu,\BA).
\end{align*}
By the duality $\nabla_\BA(\lambda)_\mu$ and $\nabla^*_\BA(\lambda)_\mu$ are free $\BA$-modules satisfying 
\[
\BF\otimes_\BA \nabla_\BA(\lambda)_\mu\cong V_\BF(\lambda)_\mu,
\quad
\BF\otimes_\BA \nabla^*_\BA(\lambda)_\mu\cong V^*_\BF(\lambda)_\mu.
\]
In particular, 
$\nabla_\BA(\lambda)$ and 
$\nabla_\BA^*(\lambda)$ 
are free $\BA$-modules satisfying 
\[
\BF\otimes_\BA \nabla_\BA(\lambda)\cong V_\BF(\lambda),
\quad
\BF\otimes_\BA \nabla^*_\BA(\lambda)\cong V^*_\BF(\lambda).
\]

\subsection{}
Let $R$ be a commutative ring equipped with $\zeta\in R^\times$, and consider 
$U_{R,\zeta}(\Gg)=R\otimes_\BA U_\BA(\Gg)$.
For a $U_{R,\zeta}(\Gg)$-module $V$ and $\mu\in X$ we set 
\[
V_\mu=\{v\in V\mid hv={\chi}_\mu(h)v\;(h\in U_{R,\zeta}(\Gh))\}.
\]
We say that 
a $U_{R,\zeta}(\Gg)$-module $V$ is integrable if it has the weight space decomposition
$V=\bigoplus_{\mu\in X}V_\mu$, and for any $v\in V$ and $i\in I$ we have 
${e}_i^{(n)}v={f}_i^{(n)}v=0$ 
for $n\gg0$.
We denote by $\Mod_{\inte}({U}_{R,\zeta}(\Gg))$ the category of integrable ${U}_{R,\zeta}(\Gg)$-modules.

Let $V$ be an integrable  $U_{R,\zeta}(\Gg)$-module.
Following \cite{Lbook} 
we define an invertible $R$-homomorphism 
$
T_i|_V=T_i:V\to V
$
for $i\in I$ 
by 
\begin{equation}
\label{eq:Ti}
{T}_iv=
\sum_{a-b+c=\langle\lambda,\alpha_i^\vee\rangle}
(-1)^b\zeta_i^{b-ac}f_i^{(a)}e_i^{(b)}f_i^{(c)}v
\qquad(v\in V_\lambda).
\end{equation}
Note $T_i=T'_{i,-1}$ in the notation of \cite{Lbook}.
We will use the following fact (see \cite[Proposition 5.3.4]{Lbook}).
\begin{lemma}
\label{lem:Ttens}
Let $i\in I$, and let
$V_1$, $V_2$ be integrable $U_{R,\zeta}(\Gg)$-modules.
\begin{itemize}
\item[(i)]
Assume that $v_1\in V_1$ satisfies 
$e_i^{(n)}v_1=0$ for any $n>0$.
Then we have
\[
(T_i|_{V_1\otimes V_2})(v_1\otimes v_2)
=T_iv_1\otimes T_iv_2
\]
for any $v_2\in V_2$.
\item[(ii)]
Assume $\zeta_i^2=1$.
Then we have 
${T}_i|_{V_1\otimes V_2}={T}_i|_{V_1}\otimes
{T}_i|_{V_2}$.
\end{itemize}
\end{lemma}
For $w\in W$  we define an $R$-homomorphism 
$
T_w|_V=T_w:V\to V
$
by 
${T}_w=
{T}_{i_1}\dots
{T}_{i_N}$,
where $w=s_{i_1}\dots s_{i_N}$ is a reduced expression of $w$.
It does not depend on the choice of a reduced expression.
Moreover, 
we have
$
{T}_wV_\lambda=V_{w\lambda}$ for $\lambda\in X$.
By \cite{Lbook} there exists an automorphism $T_w$ of the $R$-algebra $U_{R,\zeta}(\Gg)$ 
satisfying 
\begin{equation}
\label{eq:Tw}
{T}_wuv=
({T}_wu)({T}_wv)
\qquad
(u\in{U}_{R,\zeta}(\Gg), v\in V).
\end{equation}
\begin{lemma}
For a 
$U_{R,\zeta}(\Gg)$-module $V$
the following conditions are equivalent.
\begin{itemize}
\item[(a)]
$V$ has the weight decomposition
$V=\bigoplus_{\mu\in X}V_\mu$, and 
for any $v\in V$ the $U_{R,\zeta}(\Gg)$-submodule $U_{R,\zeta}(\Gg)v$ of $V$ is a finitely generated $R$-module.
\item[(b)]
$V$ is an integrable  $U_{R,\zeta}(\Gg)$-module.
\end{itemize}
\end{lemma}
\begin{proof}
The indication (a)$\Rightarrow$(b) is clear from
$e_i^{(n)}V_\lambda\subset V_{\lambda+n\alpha_i}$ and
$f_i^{(n)}V_\lambda\subset V_{\lambda-n\alpha_i}$.
Assume (b) holds.
For any $v\in V$, $w\in W$, $i\in I$ we have
$
(T_we_i^{(n)})v=(T_wf_i^{(n)})v=0
$
for $n\gg0$ by \eqref{eq:Tw}.
Hence we can deduce (a) using the PBW-type basis for $U_{R,\zeta}(\Gg)$ (see \cite{Lroots}).
\end{proof}

For $\lambda\in X^+$ we define left $U_{R,\zeta}(\Gg)$-modules 
$\Delta_{R,\zeta}(\lambda)$, $\nabla_{R,\zeta}(\lambda)$ 
and right $U_{R,\zeta}(\Gg)$-modules 
$\Delta^*_{R,\zeta}(\lambda)$, $\nabla^*_{R,\zeta}(\lambda)$ 
as the base changes of 
$\Delta_\BA(\lambda)$, $\nabla_\BA(\lambda)$,  
$\Delta^*_\BA(\lambda)$, $\nabla^*_\BA(\lambda)$
respectively.
We call 
$\Delta_{R,\zeta}(\lambda)$, $\Delta^*_{R,\zeta}(\lambda)$ 
the (left and right) Weyl modules with highest weight $\lambda$, and
$\nabla_{R,\zeta}(\lambda)$, $\nabla^*_{R,\zeta}(\lambda)$
the (left and right) dual Weyl modules with highest weight $\lambda$.

\subsection{}
In this subsection we assume  $\zeta_\alpha=\pm1$ for any $\alpha\in\Delta$.
Take a subset $J$ of $I$ as in \eqref{eq:J}, and identify $\overline{U}_R(\Gg)$ with a subalgebra of $U_{R,\zeta}(\Gg)$ via Proposition \ref{prop:pm1}.
Since $U_{R,\zeta}(\Gg)$ is generated by $\overline{U}_R(\Gg)$ and  $U_{R,\zeta}(\Gh)$, we see easily the following.
\begin{proposition}
\label{prop:rep-pm1a}
The embedding 
$\overline{U}_R(\Gg)\subset U_{R,\zeta}(\Gg)$
induces the equivalence
\[
\Mod_\inte(U_{R,\zeta}(\Gg))
\cong
\Mod_\inte(\overline{U}_R(\Gg))
\]
of abelian categories.
\end{proposition}
\begin{lemma}
\label{lem:oTT}
For $w\in W$ and $\lambda\in X$ there exists $\epsilon_{w,\lambda}\in\{\pm1\}$ such that 
for any integrable $U_{R,\zeta}(\Gg)$-module $V$ we have
\[
\overline{T}_wv=\epsilon_{w,\lambda}T_wv
\qquad(v\in V_\lambda).
\]
Here, $\overline{T}_w$ $($resp.\ 
$T_w$$)$
is defined as an operator on the integrable  $\overline{U}_{R}(\Gg)$-module 
$($resp.\
$U_{R,\zeta}(\Gg)$-module$)$ V.
\end{lemma}
\begin{proof}
We may assume $w=s_i$ for $i\in I$.
Recall
\[
T_iv=
\sum_{a-b+c=\langle\lambda,\alpha_i^\vee\rangle}
(-1)^b
\zeta_i^{b-ac}f_i^{(a)}e_i^{(b)}f_i^{(c)}v.
\]
For $i\in J$ we have
\[
e_i^{(n)}=\zeta_i^{n(n-1)/2}\overline{e}_i^{(n)},
\qquad
f_i^{(n)}=\zeta_i^{n(n+1)/2}\overline{f}_i^{(n)}k_i^{-n}
\qquad(n\in \BZ_{\geqq0}), 
\]
and hence
\begin{align*}
T_iv=
\zeta_i^{\langle\lambda,\alpha_i^\vee\rangle(\langle\lambda,\alpha_i^\vee\rangle-1)/2}
\sum_{a-b+c=\langle\lambda,\alpha_i^\vee\rangle}
(-1)^b
\overline{f}_i^{(a)}
\overline{e}_i^{(b)}
\overline{f}_i^{(c)}v
=
\zeta_i^{\langle\lambda,\alpha_i^\vee\rangle(\langle\lambda,\alpha_i^\vee\rangle-1)/2}
\overline{T_i}v.
\end{align*}
For $i\notin J$ we have
\[
e_i^{(n)}=\zeta_i^{n(n+1)/2}\overline{e}_i^{(n)}k_i^{-n},
\qquad
f_i^{(n)}=\zeta_i^{n(n-1)/2}\overline{f}_i^{(n)}
\qquad(n\in \BZ_{\geqq0}),
\] 
and hence
\begin{align*}
T_iv=
\zeta_i^{\langle\lambda,\alpha_i^\vee\rangle(\langle\lambda,\alpha_i^\vee\rangle+1)/2}
\sum_{a-b+c=\langle\lambda,\alpha_i^\vee\rangle}
(-1)^b
\overline{f}_i^{(a)}
\overline{e}_i^{(b)}
\overline{f}_i^{(c)}v
=
\zeta_i^{\langle\lambda,\alpha_i^\vee\rangle(\langle\lambda,\alpha_i^\vee\rangle+1)/2}
\overline{T_i}v.
\end{align*}
\end{proof}
\begin{proposition}
\label{prop:rep-pm1}
As left or right $\overline{U}_R(\Gg)$-modules we have
\[
\Delta_{R,\zeta}(\lambda)\cong
\overline{\Delta}_R(\lambda),
\quad
\Delta^*_{R,\zeta}(\lambda)\cong
\overline{\Delta}^*_R(\lambda),
\quad
\nabla_{R,\zeta}(\lambda)\cong
\overline{\nabla}_R(\lambda),
\quad
\nabla^*_{R,\zeta}(\lambda)\cong
\overline{\nabla}^*_R(\lambda)
\]
for $\lambda\in X^+$.
\end{proposition}
\begin{proof}
By duality we have only to show the first two isomorphisms.
The proofs being similar, we only verify the first one.
We may assume $\Gg$ is simple.
Then as in the proof of Proposition \ref{prop:pm1} we may assume $q=q_\alpha$ for short roots $\alpha$, and $R=\BZ$.
Using the embeddings
\[
\Delta_{\BZ,\pm1}(\lambda)
\subset
\Delta_{\BQ,\pm1}(\lambda),
\qquad
\overline{\Delta}_\BZ(\lambda)
\subset
\overline{\Delta}_\BQ(\lambda)
\]
the proof is reduced to showing
$\Delta_{\BQ,\pm1}(\lambda)\cong
\overline{\Delta}_\BQ(\lambda)$.
By
\[
\overline{\Delta}_\BQ(\lambda)
\cong
\overline{U}_\BQ(\Gg)/
\left(
\overline{U}_\BQ(\Gg)\Gn^+
+
\sum_{h\in\overline{U}_\BQ(\Gh)}
\overline{U}_\BQ(\Gg)(h-\overline{\chi}_\lambda(h))
+
\sum_{i\in I}
\overline{U}_\BQ(\Gg)
\overline{f}_i^
{(\langle\lambda,\alpha_i^\vee\rangle+1)}
\right)
\]
we can check that we have a surjective homomorphism 
$
\overline{\Delta}_\BQ(\lambda)
\to
\Delta_{\BQ,\pm1}(\lambda)
$
given by 
$\overline{v}_\lambda\mapsto v_\lambda$.
We conclude that this is actually an isomorphism considering the dimensions.
\end{proof}

\subsection{}
Let $R$ be a commutative ring equipped with $\zeta\in R^\times$.
In this subsection we assume that 
there exists some $\ell\in\BZ_{>0}$ such that $f_\ell(\zeta)=0$, where $f_\ell$ is the $\ell$-th cyclotomic polynomial.
Recall that we have the quantum Frobenius homomorphism
$\CF:U_{R,\zeta}(\Gg)\to U_{R,\zeta}({}^\sharp \Gg)$.
Applying Proposition \ref{prop:pm1} to $U_{R,\zeta}({}^\sharp \Gg)$ we obtain
an embedding $\overline{U}_R({}^\sharp \Gg)\subset U_{R,\zeta}({}^\sharp \Gg)$ of $R$-algebras (depending on the choice of a subset $J$ of $I$).
To avoid the confusion
we denote the left and right Weyl modules over $\overline{U}_R({}^\sharp \Gg)$ with highest weight $\lambda\in{}^\sharp X$ by 
${}^\sharp\overline{\Delta}_R(\lambda)$,
${}^\sharp\overline{\Delta}^*_R(\lambda)$,
and 
the left and right dual Weyl modules over $\overline{U}_R({}^\sharp \Gg)$ with highest weight $\lambda\in{}^\sharp X$ by 
${}^\sharp\overline{\nabla}_R(\lambda)$,
${}^\sharp\overline{\nabla}^*_R(\lambda)$ respectively.
Similarly, 
we denote the left and right Weyl modules over ${U}_{R,\zeta}({}^\sharp \Gg)$ with highest weight $\lambda\in{}^\sharp X$ by 
${}^\sharp{\Delta}_{R,\zeta}(\lambda)$,
${}^\sharp{\Delta}^*_{R,\zeta}(\lambda)$,
and 
the left and right dual Weyl modules over ${U}_{R,\zeta}({}^\sharp \Gg)$ with highest weight $\lambda\in{}^\sharp X$ by 
${}^\sharp{\nabla}_{R,\zeta}(\lambda)$,
${}^\sharp{\nabla}^*_{R,\zeta}(\lambda)$ respectively.
As in Proposition \ref{prop:rep-pm1} we can identify the left or right  $\overline{U}_R({}^\sharp\Gg)$-modules
${}^\sharp\overline{\Delta}_R(\lambda)$,
${}^\sharp\overline{\Delta}^*_R(\lambda)$,
${}^\sharp\overline{\nabla}_R(\lambda)$, 
${}^\sharp\overline{\nabla}^*_R(\lambda)$
for $\lambda\in{}^\sharp X$
with the 
left or right  $U_{R,\zeta}({}^\sharp\Gg)$-modules
${}^\sharp{\Delta}_{R,\zeta}(\lambda)$,
${}^\sharp{\Delta}^*_{R,\zeta}(\lambda)$,
${}^\sharp{\nabla}_{R,\zeta}(\lambda)$, 
${}^\sharp{\nabla}^*_{R,\zeta}(\lambda)$
respectively.
Hence they can also be regarded as 
left or right  $U_{R,\zeta}(\Gg)$-modules via $\CF$.
Under this identification
we have homomorphisms 
\begin{align*}
&\Delta_{R,\zeta}(\lambda)
\twoheadrightarrow
{}^\sharp\overline{\Delta}_R(\lambda)
\to
{}^\sharp\overline{\nabla}_R(\lambda)
\hookrightarrow
\nabla_{R,\zeta}(\lambda),
\\
&\Delta^*_{R,\zeta}(\lambda)
\twoheadrightarrow
{}^\sharp\overline{\Delta}^*_R(\lambda)
\to
{}^\sharp\overline{\nabla}^*_R(\lambda)
\hookrightarrow
\nabla^*_{R,\zeta}(\lambda)
\end{align*}
of $U_{R,\zeta}(\Gg)$-modules for $\lambda\in{}^\sharp X$.
They induce isomorphisms
\begin{align}
\label{eq:extremalA}
&\Delta_{R,\zeta}(\lambda)_{w\lambda}
\cong
{}^\sharp\overline{\Delta}_R(\lambda)_{w\lambda}
\cong
{}^\sharp\overline{\nabla}_R(\lambda)_{w\lambda}
\cong
\nabla_{R,\zeta}(\lambda)_{w\lambda},
\\
\label{eq:extremalB}
&\Delta^*_{R,\zeta}(\lambda)_{w\lambda}
\cong
{}^\sharp\overline{\Delta}^*_R(\lambda)_{w\lambda}
\cong
{}^\sharp\overline{\nabla}^*_R(\lambda)_{w\lambda}
\cong
\nabla^*_{R,\zeta}(\lambda)_{w\lambda}
\end{align}
of $R$-modules for any $w\in W$.

\section{Duality for Hopf algebras}
\subsection{}
Let $\CT$ be a Hopf algebra over a field $K$, which is commutative and cocommutative.
Then the set $\Hom_\alg(\CT,K)$ of algebra homomorphisms from $\CT$ to $K$ is endowed with a structure of abelian group by
\[
(\varphi\psi)(t)=
(\varphi\otimes\psi)(\Delta(t))
\qquad
(\varphi, \psi\in\Hom_\alg(\CT,K)).
\]
For a subgroup $\Upsilon$ of $\Hom_\alg(\CT,K)$
we denote by $\Mod^f_\Upsilon(\CT)$ 
the category of finite-dimensional $\CT$-modules $M$ with the weight space decomposition 
\[
M=
\bigoplus_{\varphi\in\Upsilon}M_\varphi,
\qquad
M_\varphi=\{m\in M\mid tm=\varphi(t)m\;(t\in \CT)\}.
\]

Assume that we are given a Hopf algebra $\CH$, which contains $\CT$ as a Hopf subalgebra.
Note that the dual space $\CH^*=\Hom_K(\CH,K)$ is endowed with an $\CH$-bimodule structure by
\[
\langle h_1fh_2, h\rangle
=\langle f, h_2hh_1\rangle
\qquad(f\in \CH^*,\; h, h_1, h_2\in \CH).
\]
By a standard argument we have the following (see for example \cite{T0}).
\begin{proposition}
\label{prop:Hdual}
The following conditions on $f\in \CH^*$ are equivalent to each other.
\begin{itemize}
\item[(a)]
$\CH f\in\Mod^f_\Upsilon(\CT)$,
\item[(b)]
$f\CH\in\Mod^f_\Upsilon(\CT)$,
\item[(c)]
$\CH f\CH\in\Mod^f_{\Upsilon\times\Upsilon}(\CT\otimes\CT)$,
\item[(d)]
there exists a two-sided ideal $\CI$ of $\CH$ such that 
$\langle f,\CI\rangle=\{0\}$ and $\CH/\CI\in
\Mod^f_{\Upsilon\times\Upsilon}(\CT\otimes \CT)$.
\end{itemize}
\end{proposition}
We denote by $\CH^*_{\CT,\Upsilon}$ the subspace of $\CH^*$ consisting of $f\in \CH^*$ satisfying the equivalent conditions of Proposition \ref{prop:Hdual}
.
Then $\CH^*_{\CT,\Upsilon}$ turns out to be a Hopf algebra whose multiplication, unit, comultiplication, counit, antipode are induced by the transpose of 
the comultiplication, the counit, the multiplication, the unit, the antipode respectively of $\CH$.

We denote by $\Mod^f_{\CT,\Upsilon}(\CH)$ the category of left $\CH$-modules which belong to $\Mod^f_\Upsilon(\CT)$ as a $\CT$-module.
More generally, we denote by 
$\Mod^{lf}_{\CT,\Upsilon}(\CH)$ the category of left $\CH$-modules which is a sum of submodules belonging to $\Mod^f_\Upsilon(\CT)$.
For $M\in \Mod^f_{\CT,\Upsilon}(\CH)$ we have a homomorphism
\begin{equation}
\Phi_M:M\otimes M^*\to \CH^*_{\CT,\Upsilon}
\end{equation}
of $\CH$-bimodules given by
\[
\langle\Phi_M(m\otimes m^*),h\rangle
=\langle m^*,hm\rangle
\qquad(m\in M, \; m^*\in M^*,\; h\in \CH).
\]
Here, $M\otimes M^*$ is regarded as an $\CH$-bimodule by
\[
h_1(m\otimes m^*)h_2
=
h_1m\otimes m^*h_2
\qquad
(m\in M, \; m^*\in M^*,\; h_1, h_2\in \CH).
\]
Denote by $\Mod^{\irr}_{\CT,\Upsilon}(\CH)$ the set of isomorphism classes of the irreducible $\CH$-modules belonging to $\Mod^f_{\CT,\Upsilon}(\CH)$.
By a standard argument we have the following (see for example \cite{T0}).
\begin{proposition}
\label{prop:Hdual2}
\begin{itemize}
\item[(i)]
We have 
\[
\CH^*_{\CT,\Upsilon}=
\sum_{M\in\Mod^f_{\CT,\Upsilon}(\CH)}\Image(\Phi_M)
=
\bigcup_{M\in\Mod^f_{\CT,\Upsilon}(\CH)}\Image(\Phi_M)
.
\]
\item[(ii)]
Assume that $\Mod^f_{\CT,\Upsilon}(\CH)$ is a semisimple category and that $\End_\CH(M)=K\id$ for any 
$M\in\Mod^{\irr}_{\CT,\Upsilon}(\CH)$.
Then 
\[
\bigoplus_{M\in\Mod^\irr_{\CT,\Upsilon}(\CH)}\Phi_M:
\bigoplus_{M\in\Mod^\irr_{\CT,\Upsilon}(\CH)}M\otimes M^*
\to \CH^*_{\CT,\Upsilon}
\]
is an isomorphism of $\CH$-bimodules.
\end{itemize}
\end{proposition}
\subsection{}
For a finite subset $\Gamma$ of $\Mod^{\irr}_{\CT,\Upsilon}(\CH)$
we define a two-sided ideal $\CI(\Gamma)$ of $\CH$ by
\[
\CI(\Gamma)=\{h\in \CH\mid
hM=\{0\}\;\;(\forall M\in \Gamma)\}.
\]
\begin{lemma}
\label{lem:ss}
Assume that $\Mod^f_{\CT,\Upsilon}(\CH)$ is a semisimple category and that $\End_\CH(M)=K\id$ for any 
$M\in\Mod^{\irr}_{\CT,\Upsilon}(\CH)$.
\begin{itemize}
\item[(i)]
For any finite subset $\Gamma$ of $\Mod^{\irr}_{\CT,\Upsilon}(\CH)$
we have
\[
\CH/\CI(\Gamma)\cong
\bigoplus_{M\in \Gamma}
\End_K(M).
\]
\item[(ii)]
For $f\in \CH^*$ we have $f\in \CH^*_{\CT,\Upsilon}$ if and only if 
there exists a finite subset $\Gamma$ of 
$\Mod^{\irr}_{\CT,\Upsilon}(\CH)$ such that
$\langle f, \CI(\Gamma)\rangle=\{0\}$.
\end{itemize}
\end{lemma}
\begin{proof}
(i) 
Note that any $M\in \Gamma$ is an irreducible $\CH/\CI(\Gamma)$-module.
Hence by a well-known fact on finite dimensional algebras the assertion follows from
\[
\dim \CH/\CI(\Gamma)\leqq\sum_{M\in\Gamma}(\dim M)^2.
\]
To verify this it is sufficient to show for finite subsets
$\Gamma$, $\Gamma'$ of $\Mod^{\irr}_{\CT,\Upsilon}(\CH)$ satisfying 
$\Gamma'=\Gamma\sqcup\{M\}$ that 
$\dim \CI(\Gamma)/\CI(\Gamma')\leqq (\dim M)^2$.
This follows from
\[
\Ker(\CI(\Gamma)\to\End_K(M))=\CI(\Gamma').
\]

(ii) Assume $f\in \CH^*_{\CT,\Upsilon}$.
For 
\[
\Gamma=\{M\in \Mod^{\irr}_{\CT,\Upsilon}(\CH)
\mid
\Hom_\CH(M,\CH f)\ne\{0\}
\]
we have
\[
\langle f,\CI(\Gamma)\rangle
=\langle \CI(\Gamma)f,1\rangle
\subset\langle \CI(\Gamma)(\CH f),1\rangle=
\{0\}.
\]
The converse is clear from (i).
\end{proof}

\subsection{}
In general, for a coalgebra $\CC$
we denote by $\Comod(\CC)$ (resp.\
 $\Comod^f(\CC)$) the category of right $\CC$-comodules 
 (resp.\ finite dimensional right $\CC$-comodules).
 
 Note that for $M\in\Mod^{f}_{\CT,\Upsilon}(\CH)$ we have a right $\CH^*_{\CT,\Upsilon}$-comodule structure
$\gamma_M:M\to M\otimes \CH^*_{T,\Upsilon}$ given by
\[
hm=\sum_a\langle f_a,h\rangle m_a\;(h\in \CH)
\;\Longrightarrow\;
\gamma_M(m)=\sum_am_a\otimes f_a.
\]
This induces functors
\begin{equation}
\label{eq:ModComod}
\Mod^{f}_{\CT,\Upsilon}(\CH)\to
\Comod^f(\CH^*_{\CT,\Upsilon}),
\quad
\Mod^{lf}_{\CT,\Upsilon}(\CH)\to
\Comod(\CH^*_{\CT,\Upsilon}).
\end{equation}

\begin{proposition}
The functors in \eqref{eq:ModComod} give equivalences of categories.
\end{proposition}
\begin{proof}
Assume that $M$ is a finite-dimensional right $\CH^*_{T,\Upsilon}$-comodule with respect to
$\gamma:M\to M\otimes \CH^*_{T,\Upsilon}$.
Then we can define a left $\CH$-module structure of $M$ by
\[
\gamma(m)=\sum m_a\otimes f_a
\;\Longrightarrow\;
hm=\sum_i\langle f_a,h\rangle m_a
\quad(h\in\CH).
\]
It is easily seen that this left $\CH$-module belongs to $\Mod^f_{\CT,\Upsilon}(\CH)$.
Moreover, this induces the inverse to the functors in \eqref{eq:ModComod}.
\end{proof}
We will sometimes identify $\Mod^{lf}_{\CT,\Upsilon}(\CH)$ with 
$\Comod(\CH^*_{\CT,\Upsilon})$.
\section{Coordinate algebras of the quantized algebraic groups}
\subsection{}
Let $K$ be a field.
We set
\[
\overline{O}_K(G)=\overline{U}_{K}(\Gg)^*_{\overline{U}
_K(\Gh),X}
.
\]
It is  isomorphic as a Hopf algebra to the coordinate algebra of the reductive algebraic group $G_K$ over $K$ with the same root datum $(X,\Delta, Y,\Delta^\vee)$ as $G$.
\subsection{}
Let $K$ be a field equipped with $\zeta\in K^\times$.

By Lemma \ref{lem:cr} the map $X\ni\lambda\mapsto\chi_\lambda\in\Hom_{\alg}(U_{K,\zeta}(\Gh),K)$ is an injective group homomorphism.
We will regard $X$ as a subgroup of $\Hom_{\alg}(U_{K,\zeta}(\Gh),K)$ in the following.
Set
\begin{gather*}
\;\;
O_{K,\zeta}(B)=
U_{K,\zeta}(\Gb)^*_{U_{K,\zeta}(\Gh),X},
\;\;
O_{K,\zeta}(H)=
U_{K,\zeta}(\Gh)^*_{U_{K,\zeta}(\Gh),X},
\\
O_{K,\zeta}(G)=
U_{K,\zeta}(\Gg)^*_{U_{K,\zeta}(\Gh),X}.
\end{gather*}
It is easily seen that
\[
O_{K,\zeta}(H)=\bigoplus_{\lambda\in X}
K\chi_\lambda.
\]
We identify 
$U_{K,\zeta}(\Gn)^*\otimes U_{K,\zeta}(\Gh)^*\otimes U_{K,\zeta}(\Gn^+)^*$
with a subspace of 
$U_{K,\zeta}(\Gg)^*$
by
\[
\langle\psi\otimes\chi\otimes\varphi,
yhx\rangle
=
\langle\psi, y\rangle
\langle\chi, h\rangle
\langle\varphi, x\rangle
\quad
(y\in U_{K,\zeta}(\Gn), 
h\in U_{K,\zeta}(\Gh), 
x\in U_{K,\zeta}(\Gn^+)).
\]
Similarly, we identify 
$U_{K,\zeta}(\Gn)^*\otimes U_{K,\zeta}(\Gh)^*$ with a subspace of 
$U_{K,\zeta}(\Gb)^*$.
Set 
\begin{align*}
U_{K,\zeta}(\Gn)^\bigstar
=&\bigoplus_{\gamma\in Q^+}(U_{K,\zeta}(\Gn)_{-\gamma})^*
\subset
U_{K,\zeta}(\Gn)^*,
\\
U_{K,\zeta}(\Gn^+)^\bigstar
=&\bigoplus_{\gamma\in Q^+}(U_{K,\zeta}(\Gn^+)_{\gamma})^*
\subset
U_{K,\zeta}(\Gn^+)^*.
\end{align*}
Then we have 
\begin{align}
\label{eq:tri-dual}
O_{K,\zeta}(G)
\subset
U_{K,\zeta}(\Gn)^\bigstar
\otimes
O_{K,\zeta}(H)\otimes
U_{K,\zeta}(\Gn^+)^\bigstar
\subset
U_{K,\zeta}(\Gg)^*.
\end{align}
Moreover, we have
\begin{align}
O_{K,\zeta}(B)
=
U_{K,\zeta}(\Gn)^\bigstar
\otimes
O_{K,\zeta}(H)
\subset
U_{K,\zeta}(\Gb)^*,
\end{align}
and the natural homomorphism
$O_{K,\zeta}(G)\to O_{K,\zeta}(B)$
is surjective (see, for example \cite[Section 2.7]{T0}).
Hence we have
\[
O_{K,\zeta}(B)\cong
O_{K,\zeta}(G)/\{f\in O_{K,\zeta}(G)\mid
\langle f,U_{K,\zeta}(\Gb)\rangle=\{0\}\}.
\]
\begin{proposition}
\label{prop:O1}
Assume $\zeta_\alpha^2=1$ for any $\alpha\in\Delta$.
The surjection
$U_{K,\zeta}(\Gg)^*\to \overline{U}_{K}(\Gg)^*$
induced by the embedding 
$\overline{U}_{K}(\Gg)\subset U_{K,\zeta}(\Gg)$
given in Proposition \ref{prop:pm1}
restricts to an isomorphism
\begin{equation}
\label{eq:O1}
O_{K,\zeta}(G)\cong \overline{O}_K(G)
\end{equation}
of $\overline{U}_{K}(\Gg)$-bimodules.
\end{proposition}
\begin{proof}
By Proposition \ref{prop:rep-pm1a} 
the surjection
$U_{K,\zeta}(\Gg)^*\to \overline{U}_{K}(\Gg)^*$
restricts to a surjective homomorphism
$O_{K,\zeta}(G)\to \overline{O}_K(G)$
of $\overline{U}_{K}(\Gg)$-bimodules.
In order to prove that it is injective
it is sufficient to show 
\[
\Ker(M\otimes M^*\to\overline{U}_{K}(\Gg)^*)
=
\Ker(M\otimes M^*\to{U}_{K,\zeta}(\Gg)^*)
\]
for any $M\in\Mod_{\inte}({U}_{K,\zeta}(\Gg))
=\Mod_{\inte}(\overline{U}_K(\Gg))$
by Proposition \ref{prop:Hdual2}.
Since $M\otimes M^*\to\overline{U}_{K}(\Gg)^*$ is a homomorphism of $\overline{U}_K(\Gg)$-bimodule, we have only to show
\[
\Ker(M_\lambda\otimes M^*\to\overline{U}_{K}(\Gg)^*)
=
\Ker(M_\lambda\otimes M^*\to{U}_{K,\zeta}(\Gg)^*)
\]
for any $\lambda\in X$.
By
$\overline{U}_{K}(\Gg)\cong
\overline{U}_{K}(\Gn)\otimes
\overline{U}_{K}(\Gn^+)\otimes
\overline{U}_{K}(\Gh)
$
we have
\[
\Ker(M_\lambda\otimes M^*\to\overline{U}_{K}(\Gg)^*)
=
\Ker(M_\lambda\otimes M^*\to
(\overline{U}_{K}(\Gn)\otimes
\overline{U}_{K}(\Gn^+))^*).
\]
Similarly, we have
\[
\Ker(M_\lambda\otimes M^*\to{U}_{K,\zeta}(\Gg)^*)
=
\Ker(M_\lambda\otimes M^*\to
({U}'_{K,\zeta}(\Gn)\otimes
{U}'_{K,\zeta}(\Gn^+))^*).
\]
They coincide under the identification $\overline{U}_{K}(\Gg)={U}'_{K,\zeta}(\Gg)$.
\end{proof}
\begin{remark}
{\rm
The isomorphism \eqref{eq:O1} does not preserve the ring structure in general $($see Remark \ref{rem:isomU}$)$.
}
\end{remark}

\subsection{}
Assume that $\zeta\in K^\times$ has the order $\ell<\infty$.
Define the root datum $({}^\sharp X,{}^\sharp\Delta,{}^\sharp Y,{}^\sharp\Delta^\vee)$ and the corresponding complex reductive group ${}^\sharp G$ as in Section \ref{sect:QE}.
The quantum Frobenius homomorphism $\CF$ induces the injection ${}^t\CF: U_{{K,\zeta}}({}^\sharp\Gg)^*\hookrightarrow U_{K,\zeta}(\Gg)^*$.
This restricts to an injective  Hopf algebra homomorphism
\begin{equation}
\label{eq:embA}
O_{K,\zeta}({}^\sharp G)\hookrightarrow O_{K,\zeta}(G).
\end{equation}

\section{Induction functor}
\subsection{}
Let $\CH$, $\CH'$ be Hopf algebras over a field $K$ with invertible antipodes, and let $p:\CH\to \CH'$ be a surjective Hopf algebra homomorphism.
We have a natural exact functor 
\[
\Res^{\CH}_{\CH'}:\Comod(\CH)\to \Comod(\CH'),
\]
where, for a $K$-module $V$ with the right $\CH$-comodule structure $\gamma:V\to V\otimes\CH$, 
we associate the right $\CH'$-comodule given by $(1\otimes p)\circ\gamma:V\to V\otimes\CH'$.
We can also define the induction functor 
\[
\Ind^{\CH}_{\CH'}:\Comod(\CH')\to\Comod(\CH)
\]
as follows.
Let $M$ be a right $\CH'$-comodule.
Regarding $\CH$ as a right $\CH'$-comodule 
via 
$
(1\otimes p)\circ\Delta:\CH\to \CH\otimes \CH'
$, 
the tensor product $M\otimes \CH$ of the two right $\CH'$-comodules is endowed with a right  $\CH'$-comodule structure
$\gamma:M\otimes \CH\to (M\otimes \CH)\otimes \CH'$.
We set
\[
\Ind^{\CH}_{\CH'}(M)
=\{n\in M\otimes \CH\mid \gamma(n)=n\otimes 1\}.
\]
Then $\Ind^{\CH}_{\CH'}(M)$ is endowed with a right $\CH$-comodule structure induced by that of 
$M\otimes \CH$
given by
\[
M\otimes \CH\to (M\otimes \CH)\otimes \CH
\qquad
(m\otimes h\mapsto
\sum_{k}(m\otimes h'_{k})
\otimes S^{-1}h_{k}),
\]
where
$\Delta(h)=\sum_kh_k\otimes h'_k$.
For $V\in\Comod(\CH)$, $M\in\Comod(\CH')$ we have
\[
\Hom(V,\Ind_{\CH'}^{\CH}(M))\cong\Hom(\Res^{\CH}_{\CH'}(V),M).
\]
It follows that 
$\Ind_{\CH'}^{\CH}$ is left exact, and we have
\[
\Ind_{\CH''}^{\CH}=\Ind_{\CH'}^{\CH}\circ \Ind_{\CH''}^{\CH'}
\]
for a sequence $\CH\to \CH'\to \CH''$ of surjective homomorphisms of Hopf algebras with invertible antipodes.
The following fact is well-known.
\begin{lemma}
\label{lem:IRT}
For $M\in\Comod(\CH')$ and $V\in\Comod(\CH)$ we have
\[
\Ind^\CH_{\CH'}(\Res^\CH_{\CH'}(V)\otimes M)
\cong
V\otimes \Ind^\CH_{\CH'}(M),
\quad
\Ind^\CH_{\CH'}(M\otimes\Res^\CH_{\CH'}(V))
\cong
\Ind^\CH_{\CH'}(M)\otimes V.
\]
\end{lemma}
\subsection{}
Assume that $K$ is a field equipped with $\zeta\in K^\times$.
We consider the left exact functor
\[
\Ind=
\Ind_{\CO_{K,\zeta}(B)}^{\CO_{K,\zeta}(G)}:
\Comod(\CO_{K,\zeta}(B))
\to
\Comod(\CO_{K,\zeta}(G)).
\]
It is known that the abelian category 
$\Comod(\CO_{K,\zeta}(B))$ has enough injectives, and we have its right derived functors
\[
R^k\Ind:\Comod(\CO_{K,\zeta}(B))
\to
\Comod(\CO_{K,\zeta}(G))
\qquad(k\geqq0).
\]
For $\lambda\in X$ we denote by $K_\lambda=K{1}_\lambda$ the one dimensional $U_{K,\zeta}(\Gb)$-module given by
\[
h1_\lambda=\chi_\lambda(h)1_\lambda
\;\;(h\in U_{K,\zeta}(\Gh)),
\quad
z1_\lambda=\varepsilon(z)1_\lambda
\;\;(z\in U_{K,\zeta}(\Gn)).
\]
We also denote the corresponding right $\CO_{K,\zeta}(B)$-comodule by $K_\lambda=K{1}_\lambda$.
The following fact is standard (see \cite{APW}).
\begin{lemma}
\label{lem:Ind}
For $\lambda\in X$ we have 
\[
\Ind(K_\lambda)
\cong
\begin{cases}
\nabla_{K,\zeta}(\lambda)\quad&(\lambda\in X^+)
\\
0&(\lambda\notin X^+).
\end{cases}
\]
\end{lemma}
The following Kempf type vanishing theorem is a consequence of the deep theory of crystal bases due to Kashiwara.
\begin{proposition}[Ryom-Hansen \cite{RH}]
\label{prop:Ind}
For $\lambda\in X^+$ and $k\ne0$
we have
$
R^k\Ind(K_\lambda)=0
$.
\end{proposition}
\section{Homogeneous coordinate algebras}
\subsection{}
For a field $K$ we set
\begin{equation}
\overline{O}_K(N\backslash G)
=\{f\in \overline{O}_K
(G)\mid
fz=\varepsilon(z)f\;(z\in \overline{U}_K(\Gn))\}.
\end{equation}
It is a subalgebra of $\overline{O}_K
(G)$ and a $(\overline{U}_K(\Gg), \overline{U}_K(\Gh))$-bimodule.
Set
\[
\overline{O}_K
(N\backslash G;\lambda)
=
\{f\in \overline{O}_K
(N\backslash G)
\mid
fh=\chi_\lambda(h)f\;(h\in \overline{U}_K(\Gh))\}
\]
for $\lambda\in X^+$.
It is known that
\begin{equation}
\overline{O}_K
(N\backslash G)
=
\bigoplus_{\lambda\in X^+}
\overline{O}_K
(N\backslash G;\lambda).
\end{equation}
Moreover, we have an isomorphism
\begin{equation}
\label{eq:oAlambda}
\overline{\nabla}_{K}(\lambda)\cong
\overline{O}_{K}(N\backslash G;\lambda)
\qquad
(v\leftrightarrow
\Phi_{\overline{\nabla}_{K}(\lambda)}
(v\otimes \overline{v}^{*}_\lambda))
\end{equation}
of left $\overline{U}_{K}(\Gg)$-modules.
Here, the image of 
$\overline{v}^*_\lambda\in
\overline{\Delta}^*_\BZ(\lambda)$
in 
$\overline{\Delta}^*_K(\lambda)=
(\overline{\nabla}_K(\lambda))^*$ is denoted by $\overline{v}^*_\lambda$ by abuse of the notation.
By
\begin{equation}
\label{eq:oA-mult}
\overline{O}_K
(N\backslash G; \lambda)\;
\overline{O}_K
(N\backslash G; \mu)
\subset
\overline{O}_K
(N\backslash G; \lambda+\mu)
\qquad(\lambda, \mu\in X^+)
\end{equation}
$\overline{O}_K
(N\backslash G)$ turns out to be a commutative $K$-algebra graded by the abelian group $X$.
Under the identification \eqref{eq:oAlambda} 
we obtain 
a non-zero homomorphism
\begin{equation}
\label{eq:onn}
\overline{\nabla}_{K}(\lambda)
\otimes
\overline{\nabla}_{K}(\mu)
\to
\overline{\nabla}_{K}(\lambda+\mu)
\end{equation}
of left $\overline{U}_{K}(\Gg)$-modules
corresponding to \eqref{eq:oA-mult}.
By
\begin{align*}
&\Hom_{\overline{U}_{K}(\Gg)}
(\overline{\nabla}_{K}(\lambda)
\otimes
\overline{\nabla}_{K}(\mu)
,
\overline{\nabla}_{K}(\lambda+\mu))
\\
\cong&
\Hom_{\overline{U}_{K}(\Gg)^{\op}}
(
\overline{\Delta}^*_{K}(\lambda+\mu),
\overline{\Delta}^*_{K}(\lambda)
\otimes
\overline{\Delta}^*_{K}(\mu)) 
\cong K
\end{align*}
\eqref{eq:onn} is the unique (up to a scalar multiple) non-zero homomorphism 
of $\overline{U}_{K}(\Gg)$-modules.

\subsection{}
Let $K$ be a field equipped with $\zeta\in K^\times$.
We set
\begin{equation}
O_{K,\zeta}(N\backslash G)
=\{f\in O_{K,\zeta}(G)\mid
fz=\varepsilon(z)f\;(z\in U_{K,\zeta}(\Gn))\}.
\end{equation}
It is a subalgebra of $O_{K,\zeta}(G)$ as well as a $(U_{K,\zeta}(\Gg),U_{K,\zeta}(\Gh))$-bimodule.
By the definition of $O_{K,\zeta}(G)$ it is easily seen that
$O_{K,\zeta}(N\backslash G)$ is a direct sum of subspaces 
\[
O_{K,\zeta}(N\backslash G;\lambda)
=
\{f\in O_{K,\zeta}(N\backslash G)
\mid
fh=\chi_\lambda(h)f\;(h\in U_{K,\zeta}(\Gh))\}
\]
for $\lambda\in X$.
Note that we have
$
O_{K,\zeta}(N\backslash G;\lambda)
\cong\Ind(K_\lambda)
$
by the definition of $\Ind$.
Hence by Proposition \ref{lem:Ind} we have 
$
O_{K,\zeta}(N\backslash G;\lambda)\ne\{0\}$ only if $\lambda\in X^+$, and hence 
\begin{equation}
O_{K,\zeta}(N\backslash G)
=
\bigoplus_{\lambda\in X^+}
O_{K,\zeta}(N\backslash G;\lambda).
\end{equation}
Moreover, we have
$O_{K,\zeta}(N\backslash G;\lambda)
\cong \nabla_{K,\zeta}(\lambda)$ 
 for $\lambda\in X^+$.
More precisely, we have an isomorphism
\begin{equation}
\label{eq:Alambda}
\nabla_{K,\zeta}(\lambda)\cong
O_{K,\zeta}(N\backslash G;\lambda)
\qquad
(v\leftrightarrow
\Phi_{\nabla_{K,\zeta}(\lambda)}
(v\otimes v^{*}_\lambda))
\end{equation}
of $U_{K,\zeta}(\Gg)$-modules.
Here, the image of 
${v}^*_\lambda\in
{\Delta}^*_\BA(\lambda)$
in 
${\Delta}^*_{K,\zeta}(\lambda)=
(\nabla_{K,\zeta}(\lambda))^*$ is denoted by ${v}^*_\lambda$ by abuse of the notation.

By 
\begin{equation}
\label{eq:A-mult}
O_{K,\zeta}(N\backslash G; \lambda)\;
O_{K,\zeta}(N\backslash G; \mu)
\subset
O_{K,\zeta}(N\backslash G; \lambda+\mu)
\qquad(\lambda, \mu\in X^+)
\end{equation}
$O_{K,\zeta}(N\backslash G)$ turns out to be a $K$-algebra graded by the abelian group $X$.
Similarly to the case of $\overline{O}_K(N\backslash G)$, 
the multiplication 
\[
O_{K,\zeta}(N\backslash G; \lambda)\otimes
O_{K,\zeta}(N\backslash G; \mu)
\to
O_{K,\zeta}(N\backslash G; \lambda+\mu)
\]
corresponds to the unique (up to a non-zero scalar multiple) non-zero
homomorphism
\begin{equation}
\label{eq:nn}
\nabla_{K,\zeta}(\lambda)
\otimes
\nabla_{K,\zeta}(\mu)
\to
\nabla_{K,\zeta}(\lambda+\mu)
\end{equation}
of left $U_{K,\zeta}(\Gg)$-modules
under the identification \eqref{eq:Alambda}.

\subsection{}
Assume that $\zeta\in K^\times$ has the multiplicative order $\ell<\infty$.
Define a subalgebra 
$O_{K,\zeta}({}^\sharp N\backslash {}^\sharp G)$ 
of $O_{K,\zeta}({}^\sharp G)$ and its subspace 
$O_{K,\zeta}({}^\sharp N\backslash {}^\sharp G;\lambda)$ 
for $\lambda\in {}^\sharp X$ 
similarly to the case of $O_{K,\zeta}(G)$, 
so that
\[
O_{K,\zeta}({}^\sharp N\backslash {}^\sharp G)
=
\bigoplus_{\lambda\in{}^\sharp X^+}
O_{K,\zeta}({}^\sharp N\backslash {}^\sharp G;\lambda).
\]
It is easily seen that the embedding
\eqref{eq:embA} induces the embedding
\begin{equation}
O_{K,\zeta}({}^\sharp N\backslash {}^\sharp G)
\subset
O_{K,\zeta}(N\backslash G)
\end{equation}
of algebras so that
$O_{K,\zeta}({}^\sharp N\backslash {}^\sharp G;\lambda)
\subset
O_{K,\zeta}(N\backslash G;\lambda)$ 
for $\lambda\in{}^\sharp X$.
Note also that we have an isomorphism
$
O_{K,\zeta}({}^\sharp G)\cong \overline{O}_K({}^\sharp G)
$
of $\overline{U}_{K}({}^\sharp \Gg)$-bimodules
by Proposition \ref{prop:O1}.
It induces 
an isomorphism 
\begin{equation}
O_{K,\zeta}({}^\sharp N\backslash{}^\sharp G)\cong \overline{O}_K({}^\sharp N\backslash{}^\sharp G)
\end{equation}
of 
left $\overline{U}_{K}({}^\sharp \Gg)$-modules.
\begin{proposition}
\label{prop:XY}
For any $\mu\in{}^\sharp X^+$ there exists some $N\geqq0$ such that for $\lambda\in X^+$ satisfying $\langle\lambda,\alpha_i^\vee\rangle\geqq N$ for any 
$i\in I$ we have
\[
O_{K,\zeta}(N\backslash G;\lambda)
O_{K,\zeta}({}^\sharp N\backslash{}^\sharp G;\mu)
=
O_{K,\zeta}(N\backslash G;\lambda+\mu).
\]
\end{proposition}
\begin{proof}
Recall 
\[
O_{K,\zeta}(N\backslash G;\lambda)
\cong
\nabla_{K,\zeta}(\lambda)\quad(\lambda\in X^+),
\quad
O_{K,\zeta}({}^\sharp N\backslash {}^\sharp G;\mu)
\cong
{}^\sharp \nabla_{K,\zeta}(\mu)\quad(\mu\in {}^\sharp X^+).
\]
Under this identification, the multiplication 
\[
O_{K,\zeta}(N\backslash G;\lambda)
\otimes
O_{K,\zeta}({}^\sharp N\backslash {}^\sharp G;\mu)
\to
O_{K,\zeta}(N\backslash G;\lambda+\mu)
\]
corresponds to a non-zero homomorphism 
\begin{equation}
\label{eq:nabla}
\nabla_{K,\zeta}(\lambda)
\otimes
{}^\sharp \nabla_{K,\zeta}(\mu)
\to
\nabla_{K,\zeta}(\lambda+\mu)
\end{equation}
of $U_{K,\zeta}(\Gg)$-modules, where 
${}^\sharp \nabla_{K,\zeta}(\mu)$ is regarded as a 
$U_{K,\zeta}(\Gg)$-module through the quantum Frobenius homomorphism
$\CF:U_{K,\zeta}(\Gg)\to{}^\sharp U_{K,\zeta}(\Gg)$.
Note that \eqref{eq:nabla} is obtained by taking the dual of the unique (up to a scalar multiple) non-zero homomorphism
\[
\Delta^*_{K,\zeta}(\lambda+\mu)
\to
\Delta^*_{K,\zeta}(\lambda)
\otimes
{}^\sharp \Delta^*_{K,\zeta}(\mu)
\]
of right $U_{K,\zeta}(\Gg)$-modules.

Assume 
$\langle\lambda,\alpha_i^\vee\rangle\gg0$ for any 
$i\in I$.
We have a filtration 
\[
K_\lambda\otimes {}^\sharp \nabla_{K,\zeta}(\mu)
=M_1\supset M_2\supset \dots\supset M_s\supset M_{s+1}=0
\]
of the $U_{K,\zeta}(\Gb)$-module $K_\lambda\otimes {}^\sharp \nabla_{K,\zeta}(\mu)$ such that $M_j/M_{j+1}\cong K_{\lambda+\nu_j}$, where $\nu_1=\mu$, $\nu_2$, \dots, $\nu_s$ are the weights of ${}^\sharp \nabla_{K,\zeta}(\mu)$ with multiplicity.
From the short exact sequence
\[
0\to M_2\to K_\lambda\otimes {}^\sharp \nabla_{K,\zeta}(\mu)\to K_{\lambda+\mu}\to 0
\]
we obtain an exact sequence
\[
\Ind(K_\lambda\otimes {}^\sharp \nabla_{K,\zeta}(\mu))
\to 
\Ind(K_{\lambda+\mu})
\to
R^1\Ind(M_2)
\]
of $U_{K,\zeta}(\Gg)$-modules.
By Proposition \ref{prop:Ind} and Lemma \ref{lem:IRT}
we have
\[
\Ind(K_\lambda\otimes {}^\sharp \nabla_{K,\zeta}(\mu))
\cong
\nabla_{K,\zeta}(\lambda)
\otimes
{}^\sharp \nabla_{K,\zeta}(\mu),
\qquad
\Ind(K_{\lambda+\mu})
\cong
\nabla_{K,\zeta}(\lambda+\mu).
\]
By our assumption on $\lambda$ we have
$\lambda+\nu_j\in X^+$ for any $j$, and hence 
$R^k\Ind(M_j/M_{j+1})=0$ for any $j\geqq2$, $k\geqq1$.
Therefore, we have $R^1\Ind(M_2)=0$.
We obtain a surjective homomorphism
\[
\nabla_{K,\zeta}(\lambda)
\otimes
{}^\sharp \nabla_{K,\zeta}(\mu)
\to
\nabla_{K,\zeta}(\lambda+\mu)
\]
of $U_{K,\zeta}(\Gg)$-modules.
Since \eqref{eq:nabla} is the unique (up to a scalar multiple) non-zero homomorphism of $U_{K,\zeta}(\Gg)$-modules, we obtain the surjectivity of \eqref{eq:nabla}.
\end{proof}

\section{quantized flag manifolds}
\subsection{}
We assume that $G$ is semisimple and simply-connected.
Namely, we assume
$Y=Q^\vee$, so that the canonical homomorphism $X\to\Hom_\BZ(Q^\vee,\BZ)$ is bijective. 
Let $K$ be a field, and
let $G_K$ be the split semisimple algebraic group defined over $K$ with the same root datum $(X,\Delta, Y,\Delta^\vee)$ as $G$.
We denote by $B_K$, $N_K$, $N_K^+$ the subgroups of $G_K$ similarly defined as $B$, $N$, $N^+$ respectively.
Set $\CB_K=B_K\backslash G_K$.
It is a projective  algebraic variety defined over $K$, called the 
flag variety.
Denote by $\Mod(\CO_{\CB_K})$ the category of quasi-coherent $\CO_{\CB_K}$-modules.
It can be described using  $\overline{O}_K
(N\backslash G)$ as follows.
Let  $\Mod_{\gr}(\overline{O}_K
(N\backslash G))$ be the category of graded $\overline{O}_K
(N\backslash G)$-modules.
We denote by $\Tor_{\gr}(
\overline{O}_K(N\backslash G))
$ the full subcategory of $\Mod_{\gr}(\overline{O}_K(N\backslash G))$ consisting of $M\in\Mod_{\gr}(\overline{O}_K
(N\backslash G))$ such that for any $m\in M$ there exists some positive integer $N$ satisfying
\[
\lambda\in X,\;\langle\lambda,\alpha_i^\vee\rangle\geqq N
\;\;
(\forall i\in I)
\;
\Longrightarrow
\;
\overline{O}_K(N\backslash G; \lambda)m=\{0\}.
\]
Note that
$\Tor_{\gr}(\overline{O}_K
(N\backslash G))$
is closed under taking subquotients and extensions in 
$\Mod_{\gr}(\overline{O}_K
(N\backslash G))$.
Then we have
\[
\Mod(\CO_{\CB_K})
\cong
\Mod_{\gr}(\overline{O}_K
(N\backslash G))/
\Tor_{\gr}(\overline{O}_K
(N\backslash G))
:=
\overline{\CP}^{-1}\Mod_{\gr}(\overline{O}_K
(N\backslash G)),
\]
where $\overline{\CP}$ consists of morphisms in $\Mod_{\gr}(\overline{O}_K
(N\backslash G))$ whose kernel and cokernel belong to $\Tor_{\gr}(\overline{O}_K
(N\backslash G))$, and 
$\overline{\CP}^{-1}\Mod_{\gr}(\overline{O}_K
(N\backslash G))$
denotes the localization of the category so that the morphisms in $\overline{\CP}$ turn out to be isomorphisms (see \cite{GZ}, \cite{P}).

For $w\in W$ and $\lambda\in X^+$ we set
\[
\overline{\sigma}_\lambda^w
=\Phi_{\overline{\nabla}_K(\lambda)}
(\overline{T}_w\overline{v}_{\lambda}
\otimes \overline{v}^*_\lambda)
\in \overline{O}_K(N\backslash G;\lambda)
\subset \overline{O}_K
(N\backslash G).
\]
Here, the images of 
$\overline{v}_\lambda
\in
\overline{\nabla}_\BZ(\lambda)
$
and
$\overline{v}^*_\lambda
\in
\overline{\Delta}_\BZ^*(\lambda)
$
in 
$
\overline{\nabla}_K(\lambda)
$
and
$
\overline{\Delta}_K^*(\lambda)
=
(\overline{\nabla}_K(\lambda))^*
$
are denoted by 
$\overline{v}_\lambda$
and
$\overline{v}^*_\lambda$
respectively.
We have
\[
\overline{\sigma}_\lambda^w
\overline{\sigma}_\mu^w
=
\overline{\sigma}_{\lambda+\mu}^w
\qquad
(\lambda, \mu\in X^+, w\in W),
\]
and hence 
$\overline{\CS}_K^w
=
\{\overline{\sigma}_\lambda^w\mid\lambda\in X^+\}
$ is a homogeneous multiplicative subset of the commutative graded ring $\overline{O}_K(N\backslash G)$.
Therefore, the localization 
$(\overline{\CS}_K^w)^{-1}\overline{O}_K(N\backslash G)$ turns out to be a commutative graded ring graded by $X$.
Set
\[
\overline{O}_K(\CB^w)
=
((\overline{\CS}_K^w)^{-1}
\overline{O}_K(N\backslash G))
(0).
\]
This commutative ring is naturally identified with the coordinate algebra of the affine open subset
$
\CB^w_K=B_K\backslash B_KN_K^+w^{-1}
$
of $\CB_K$.
In particular, the category $\Mod(\CO_{\CB^w_K})$ of quasi-coherent $\CO_{\CB^w_K}$-modules is isomorphic to the category 
$\Mod(
\overline{O}_K(\CB^w))$ 
of $\overline{O}_K(\CB^w)$-modules.
The natural exact functor 
\[
\overline{\res}_w:\Mod(\CO_{\CB_K})\to\Mod(\CO_{\CB^w_K})
\]
given by the restriction of quasi-coherent $\CO$-modules is induced by
\[
\Mod_{\gr}(\overline{O}_K
(N\backslash G))
\ni
M
\mapsto
((\overline{\CS}_K^w)^{-1}M)(0)
\in
\Mod(
\overline{O}_K(\CB^w)).
\]

By $\CB_K=\bigcup_{w\in W}\CB^w_K$ we have obviously
\begin{equation}
M\in \Mod(\CO_{\CB_K}),\;
\overline{\res}_w(M)=0\;(\forall w\in W)
\;
\Longrightarrow
\;
M=0.
\end{equation}
It is easily seen that this is equivalent to the following.
\begin{lemma}
\label{lem:A}
For any $\mu\in X^+$ we have  
\begin{equation}
\overline{O}_{K}(N\backslash G;\lambda+\mu)=\sum_{w\in W}
\overline{O}_{K}(N\backslash G;\lambda)
\,
\overline{\sigma}_\mu^w
\end{equation}
for $\lambda\in X^+$ such that $\langle\lambda,\alpha_i^\vee\rangle\gg0$ for any $i\in I$.
\end{lemma}
\begin{remark}
\label{rem:A}
{\rm
Lemma \ref{lem:A} holds for general root data $(X,\Delta, Y,\Delta^\vee)$ without the assumption $Y=Q^\vee$.
The general case easily follows from the special case.
}
\end{remark}

\subsection{}
We continue to assume $Y=Q^\vee$.
Let $K$ be a field equipped with $\zeta\in K^\times$.
We define an abelian category 
$\Mod(\CO_{\CB_{K,\zeta}})$
by
\begin{align}
\label{eq:qs1}
\Mod(\CO_{\CB_{K,\zeta}})
:=&
\Mod_{\gr}({O}_{K,\zeta}
(N\backslash G))/
\Tor_{\gr}({O}_{K,\zeta}
(N\backslash G))
\\
\nonumber
=&
\CP^{-1}\Mod_{\gr}({O}_{K,\zeta}
(N\backslash G)).
\end{align}
Here, 
$\Mod_{\gr}({O}_{K,\zeta}
(N\backslash G))$ is the category of graded ${O}_{K,\zeta}(N\backslash G)$-modules, and 
$\Tor_{\gr}(
O_{K,\zeta}(N\backslash G))$ is its 
 full subcategory consisting of 
 $M\in\Mod_{\gr}({O}_{K,\zeta}
(N\backslash G))$ such that for any $m\in M$ there exists some positive integer $N$ satisfying
\[
\lambda\in X,\;\langle\lambda,\alpha_i^\vee\rangle\geqq N\;
\Longrightarrow
\;
O_{K,\zeta}(N\backslash G; \lambda)m=\{0\}.
\]
Moreover, $\CP$ consists of morphisms in $\Mod_{\gr}({O}_{K,\zeta}
(N\backslash G))$ whose kernel and cokernel belong to $\Tor_{\gr}({O}_{K,\zeta}
(N\backslash G))$, and 
$\CP^{-1}\Mod_{\gr}({O}_{K,\zeta}
(N\backslash G))$
denotes the localization of the category so that the morphisms in $\CP$ turn out to be isomorphisms.

For $w\in W$ and $\lambda\in X^+$ we set
\[
\sigma_\lambda^w=
\Phi_{\nabla_{K,\zeta}(\lambda)}
(T_wv_{\lambda}\otimes v^{*}_\lambda)
\in O_{K,\zeta}(N\backslash G)\subset O_{K,\zeta}(G).
\]
Here, the images of 
${v}_\lambda
\in
{\nabla}_\BA(\lambda)
$
and
${v}^*_\lambda
\in
{\Delta}_\BA^*(\lambda)
$
in 
$
{\nabla}_{K,\zeta}(\lambda)
$
and
$
{\Delta}_{K,\zeta}^*(\lambda)
=
({\nabla}_{K,\zeta}(\lambda))^*
$
are denoted by 
${v}_\lambda$
and
${v}^*_\lambda$
respectively.
We have
\[
\sigma_\lambda^w\sigma_\mu^w=
\sigma_{\lambda+\mu}^w
\qquad
(\lambda, \mu\in X^+, w\in W).
\]
Set $
\CS_{K,\zeta}^w=
\{\sigma_\lambda^w
\mid
\lambda\in X^+\}
$ for $w\in W$.
It is known that 
the multiplicative set
$
\CS_{K,\zeta}^w
$
satisfies the left and right Ore conditions in the ring
$O_{K,\zeta}(N\backslash G)$ 
(see \cite{Jo}, \cite{T0}).
Hence the localization 
$(\CS_{K,\zeta}^w)^{-1}O_{K,\zeta}(N\backslash G)$ is a graded ring graded by $X$.
We set
\[
O_{K,\zeta}(\CB^w):=
((\CS_{K,\zeta}^w)^{-1}O_{K,\zeta}(N\backslash G))
(0).
\]
\begin{remark}
{\rm
In the case $w=1$ we can identify the $K$-algebra $O_{K,\zeta}(\CB^1)$ with a subalgebra of $\CO_{K,\zeta}(B^+)$ (see \cite[Proposition 4.5]{T0}).
For general $w$ we do not know such an explicit description of the algebra 
$O_{K,\zeta}(\CB^w)$.
In fact the ring  $O_{K,\zeta}(\CB^w)$ does depend on the choice of $w\in W$.
}
\end{remark}

We define an abelian category 
$\Mod(\CO_{\CB^w_{K,\zeta}})$
to be the category of left
$O_{K,\zeta}(\CB^w)$-modules;
\begin{equation}
\label{eq:qs2}
\Mod(\CO_{\CB^w_{K,\zeta}}):=
\Mod(O_{K,\zeta}(\CB^w)).
\end{equation}
Then we have a natural exact functor
\begin{equation}
\label{eq:qs3}
{\res}_w:\Mod(\CO_{\CB_{K,\zeta}})\to\Mod(\CO_{\CB^w_{K,\zeta}})
\end{equation}
induced by
\[
\Mod_{\gr}({O}_{K,\zeta}
(N\backslash G))
\ni
M
\mapsto
((\overline{\CS}_{K,\zeta}^w)^{-1}M)(0)
\in
\Mod(\CO_{\CB^w_{K,\zeta}}).
\]
\begin{example}
{\rm
Let $G=SL_2(\BC)$.
Let $\alpha$ be the unique positive root, and set 
$\rho=\alpha/2$.
Then we have $X=\BZ\rho$, $X^+=\BZ_{\geqq0}\rho$.
In this case it is well-known that the $X$-graded $K$-algebra 
$O_{K,\zeta}(N\backslash G)$ is generated by 
the elements $a_\zeta$, $b_\zeta$ of degree $\rho$ satisfying the fundamental relation
$
a_\zeta b_\zeta=\zeta b_\zeta a_\zeta
$.
We define a functor
$F:\Mod_{\gr}(O_{K,\zeta}(N\backslash G))
\to
\Mod_{\gr}(O_{K,1}(N\backslash G))
$ as follows.
For $M=\bigoplus_{n\in\BZ}M(n\rho)\in \Mod_{\gr}(O_{K,\zeta}(N\backslash G))$ we
set $F(M)=M$ as an $X$-graded $K$-module and define the action of $O_{K,1}(N\backslash G)$ on $F(M)=M$ by
\[
a_1m=a_\zeta m,
\qquad
b_1m=\zeta^nb_\zeta m
\qquad(m\in M(n\rho)).
\]
It is easily seen that this gives equivalences 
\begin{gather*}
\Mod_{\gr}(O_{K,\zeta}(N\backslash G))
\cong
\Mod_{\gr}(O_{K,1}(N\backslash G)),
\\
\Mod(\CO_{\CB_{K,\zeta}})
\cong
\Mod(\CO_{\CB_{K}}),\qquad
\Mod(\CO_{\CB^w_{K,\zeta}})
\cong
\Mod(\CO_{\CB^w_{K}})
\end{gather*}
of categories.
}
\end{example}

The main result of this paper is the following.
\begin{theorem}
\label{thm:main}
Recall that $Y=Q^\vee$.
Assume that $\zeta$ is transcendental over the prime field $K_0$ of $K$, or the multiplicative order of $\zeta$ is finite.
Then we have
\begin{equation}
M\in \Mod(\CO_{\CB_{K,\zeta}}),\;
{\res}_w(M)=0\;(\forall w\in W)
\;
\Longrightarrow
\;
M=0.
\end{equation}
\end{theorem}

\subsection{}
We recall the notion of quasi-schemes following Rosenberg \cite{R}.
Let $\GA$, $\GB$ be abelian categories.
We say that a morphism 
$f:X_{\mathfrak{A}}\to X_{\mathfrak{B}}$
between the corresponding virtual spaces is given if we are given an isomorphism class of a right exact functor $f^*:\GB\to\GA$ admitting a right adjoint $f_*:\GA\to\GB$. 
If moreover $f_*$ is exact and faithful, we say that $f$ is an affine morphism.
Let $\GA$, $\GA_\lambda$ ($\lambda\in\Lambda$) be abelian categories and assume that we are given morphisms $f_\lambda:X_{\GA_\lambda}\to X_\GA$.
We say that $\{f_\lambda\}_{\lambda\in\Lambda}$ is a Zariski cover of $X_\GA$ if each $f^*_\lambda$ is a localization of categories in the sense of \cite{GZ} and the patching property 
\[
M\in\GA,\quad
f_\lambda^*M=0\;\;(\forall\lambda\in\Lambda)\;
\Longrightarrow\;
M=0
\]
is satisfied.
Let $\GA$, $\GB$ be abelian categories, and assume that we are given a morphism 
$f:X_{\mathfrak{A}}\to X_{\mathfrak{B}}$
between the corresponding virtual spaces.
Then we say that $X_\GA$ is a quasi-scheme over $X_\GB$ if there exists a Zariski cover 
$\{f_\lambda\}_{\lambda\in\Lambda}$ of $X_\GA$ 
such that $f_\lambda$, $f\circ f_\lambda$ are affine.
Note that
an ordinary scheme $Y$ is regarded as a quasi-scheme $X_{\Mod(\CO_Y)}$.

By Theorem \ref{thm:main} we have the following.
\begin{corollary}
Recall that $Y=Q^\vee$.
Assume that $\zeta$ is transcendental over the prime field $K_0$ of $K$, or the multiplicative order of $\zeta$ is finite.
Then 
\eqref{eq:qs1}, \eqref{eq:qs2}, \eqref{eq:qs3} 
give a quasi-scheme $\CB_{K,\zeta}$ over $\Spec(K)$ with affine open covering $\CB_{K,\zeta}=\bigcup_{w\in W}\CB^w_{K,\zeta}$, in the sense of Rosenberg \cite{R}.
\end{corollary}

\subsection{}
We no longer assume $Y=Q^\vee$.
Theorem \ref{thm:main} follows easily from the following Theorem applied to the special case $Y=Q^\vee$. 
\begin{theorem}
\label{thm:main2}
Assume that $\zeta$ is  transcendental over the prime field $K_0$ of $K$, or the multiplicative order of $\zeta$ is finite.
For any $\mu\in X^+$ we have  
\begin{equation}
O_{K,\zeta}(N\backslash G;\lambda+\mu)=\sum_{w\in W}O_{K,\zeta}(N\backslash G;\lambda)
\,
\sigma_\mu^w
\end{equation}
for $\lambda\in X^+$ such that $\langle\lambda,\alpha_i^\vee\rangle\gg0$ for any $i\in I$.
\end{theorem}

From Theorem \ref{thm:main2} we obtain the following.
\begin{corollary}
\label{cor:main2}
Assume that $\zeta$ is  transcendental over the prime field $K_0$ of $K$, or the multiplicative order of $\zeta$ is finite.
For any $\mu\in X^+$ we have  
\[
O_{K,\zeta}(G)=\sum_{w\in W}
O_{K,\zeta}(G)
\,
\sigma_\mu^w.
\]
\end{corollary}
\begin{proof}
We first show 
\begin{equation}
\label{eq:OG}
O_{K,\zeta}(G)=
O_{K,\zeta}(G)O_{K,\zeta}(N\backslash G;\nu)
\end{equation}
for any $\nu\in X^+$.
Take $\varphi\in O_{K,\zeta}(N\backslash G;\nu)$ such that $\varepsilon(\varphi)=1$.
Then we have
\[
\sum_{k}
(S^{-1}\varphi'_{k})\varphi_{k}=\varepsilon(\varphi)=1,
\]
where $\Delta(\varphi)=\sum_k\varphi_k\otimes\varphi'_k$.
By the definition of $O_{K,\zeta}(N\backslash G;\nu)$ we have
\[
\Delta(O_{K,\zeta}(N\backslash G;\nu))
\subset O_{K,\zeta}(N\backslash G;\nu)
\otimes O_{K,\zeta}(G).
\]
Hence we have
$1\in O_{K,\zeta}(G)O_{K,\zeta}(N\backslash G;\nu)$, from which we obtain \eqref{eq:OG}.

Now let $\mu\in X^+$.
By Theorem \ref{thm:main2} there exists some $\lambda\in X^+$ such that 
\[
O_{K,\zeta}(N\backslash G;\lambda+\mu)=\sum_{w\in W}O_{K,\zeta}(N\backslash G;\lambda)
\,
\sigma_\mu^w.
\]
Then we obtain
\begin{align*}
O_{K,\zeta}(G)=&
O_{K,\zeta}(G)O_{K,\zeta}(N\backslash G;\lambda+\mu)
=
\sum_{w\in W}O_{K,\zeta}(G)
O_{K,\zeta}(N\backslash G;\lambda)
\,
\sigma_\mu^w
\\
\subset&
\sum_{w\in W}O_{K,\zeta}(G)
\,
\sigma_\mu^w
.
\end{align*}
\end{proof}

\subsection{}
The remainder of this section is devoted to the proof of Theorem \ref{thm:main2}.
We will derive it from Lemma \ref{lem:A}.

We first consider the case $\zeta$ is transcendental over the prime field $K_0$ of $K$.
This case was already dealt with in \cite{Jo0} and \cite{LR}.
We include its proof here for the sake of the readers.

Set $\CR=K_0[q,q^{-1}]$.
For $\lambda\in X^+$ we define a subspace
$
O_{\CR,q}(N\backslash G;\lambda)
$
of $\Hom_\CR(U_{\CR,q}(\Gg),\CR)$ by
\[
O_{\CR,q}(N\backslash G;\lambda)
=
\{
\Phi_{\nabla_{\CR,q}(\lambda)}(v\otimes v_\lambda^{*})
\mid
v\in \nabla_{\CR,q}(\lambda)
\},
\]
where
\[
\langle
\Phi_{\nabla_{\CR,q}(\lambda)}(v\otimes v_\lambda^{*}),
u
\rangle
=
\langle
v_\lambda^{*}, uv
\rangle\qquad
(u\in U_{\CR,q}(\Gg)).
\]
We set
\[
O_{\CR,q}(N\backslash G)
=
\bigoplus_{\lambda\in X^+}
O_{\CR,q}(N\backslash G;\lambda).
\]
It is a ring graded by $X$.
Moreover, we have 
\[
K\otimes_\CR O_{\CR,q}(N\backslash G)
\cong O_{K,\zeta}(N\backslash G)
\]
with respect to $s_\zeta: \CR\to K$\;\;($q\mapsto\zeta)$, and 
\[
K_0\otimes_\CR O_{\CR,q}(N\backslash G)
\cong\overline{O}_{K_0}(N\backslash G)
\]
with respect to $s_1: \CR\to K_0$\;\;($q\mapsto1)$.
Set 
\[
\widehat{\sigma}^w_\mu=\Phi_{\nabla_{\CR,q}(\mu)}(T_wv_{\mu}\otimes v_\mu^{*})\in
O_{\CR,q}(N\backslash G)
\]
for $\mu\in X^+$, and 
consider the $\CR$-linear map
\[
F_\lambda:
O_{\CR,q}(N\backslash G;\lambda)^{\oplus W}
\to
O_{\CR,q}(N\backslash G;\lambda+\mu)
\qquad
((f_w)_{w\in W}\mapsto
\sum_{w\in W}f_w
\widehat{\sigma}^w_\mu)
\]
between free $\CR$-modules of finite rank.
Assume $\langle\lambda,\alpha_i^\vee\rangle\gg0$.
We see by Lemma \ref{lem:A} and Remark \ref{rem:A} that
$
K_0\otimes_\CR F_\lambda
$ 
 with respect to  $s_1$ is surjective.
Hence 
$
K\otimes_\CR F_\lambda
$ with respect to  $s_\zeta$ is surjective when $\zeta$ is transcendental over $K_0$.
This is exactly what we need to show.
The proof of Theorem \ref{thm:main2} is now complete in the case $\zeta$ is transcendental over the prime field $K_0$ of $K$.
\subsection{}
\label{subsection:X}
We consider the case $\zeta_\alpha^2=1$ for any $\alpha\in\Delta$.
Let $\lambda, \mu\in X^+$.
Let 
\[
\overline{\Xi}:
\overline{\nabla}_K(\lambda)
\otimes
\overline{\nabla}_K(\mu)
\to
\overline{\nabla}_K(\lambda+\mu)
\]
and
\[
{\Xi}:
{\nabla}_{K,\zeta}(\lambda)
\otimes
{\nabla}_{K,\zeta}(\mu)
\to
{\nabla}_{K,\zeta}(\lambda+\mu)
\]
be the unique (up to a non-zero scalar multiple) non-zero homomorphisms of $\overline{U}_K(\Gg)$-modules and $U_{K,\zeta}(\Gg)$-modules respectively.
In view of \eqref{eq:oAlambda}, \eqref{eq:Alambda} and Lemma \ref{lem:A} (see also Remark \ref{rem:A}) we have only to show that 
$\Xi({\nabla}_{K,\zeta}(\lambda)\otimes T_wv_\mu)$ coincides with
$\overline{\Xi}(\overline{\nabla}_{K}(\lambda)\otimes \overline{T}_w\overline{v}_\mu)$ 
under the identification 
$
{\nabla}_{K,\zeta}(\lambda+\mu)
\cong
\overline{\nabla}_{K}(\lambda+\mu)
$
of Proposition \ref{prop:rep-pm1}.
By Lemma \ref{lem:oTtens} and Lemma \ref{lem:Ttens} (ii) we have
\begin{gather*}
\overline{\Xi}(\overline{\nabla}_{K}(\lambda)\otimes \overline{T}_{w}\overline{v}_\mu)
=
\overline{\Xi}(\overline{T}_{w}(
\overline{\nabla}_{K}(\lambda)\otimes \overline{v}_\mu))
=
\overline{T}_{w}\overline{\Xi}(
\overline{\nabla}_{K}(\lambda)\otimes \overline{v}_\mu),
\\
\Xi({\nabla}_{K,\zeta}(\lambda)\otimes T_{w}v_\mu)
=
\Xi( T_{w}({\nabla}_{K,\zeta}(\lambda)\otimes v_\mu))
=
T_{w}\Xi({\nabla}_{K,\zeta}(\lambda)\otimes v_\mu).
\end{gather*}
Hence by Lemma \ref{lem:oTT}
it is sufficient to show
\begin{equation}
\label{eq:pp}
\overline{\Xi}(
\overline{\nabla}_{K}(\lambda)\otimes \overline{v}_\mu)
=
\Xi({\nabla}_{K,\zeta}(\lambda)\otimes v_\mu).
\end{equation}
Setting
\begin{align*}
M=&
\{
m\in \Delta^*_{K,\zeta}(\lambda+\mu)
\mid
\langle m, \Xi({\nabla}_{K,\zeta}(\lambda)\otimes v_\mu)\rangle
=\{0\}\},
\\
\overline{M}=&
\{
\overline{m}\in \overline{\Delta}^*_{K}(\lambda+\mu)
\mid
\langle \overline{m}, \overline{\Xi}(
\overline{\nabla}_{K}(\lambda)\otimes \overline{v}_\mu)\rangle
=\{0\}\},
\end{align*}
\eqref{eq:pp} is equivalent to $\overline{M}=M$ under the identification 
$\Delta^*_{K,\zeta}(\lambda+\mu)
=
\overline{\Delta}^*_{K}(\lambda+\mu)$.
Let 
\[
\overline{\Xi}^*:
\overline{\Delta}^*_K(\lambda+\mu)
\to
\overline{\Delta}^*_K(\lambda)
\otimes
\overline{\Delta}^*_K(\mu)
\]
and
\[
{\Xi}^*:
{\Delta}^*_{K,\zeta}(\lambda+\mu)
\to
{\Delta}^*_{K,\zeta}(\lambda)
\otimes
{\Delta}^*_{K,\zeta}(\mu)
\]
be the unique (up to a scalar multiple) non-zero homomorphisms of $\overline{U}_K(\Gg)$-modules and $U_{K,\zeta}(\Gg)$-modules respectively.
Note that any $m\in \Delta^*_{K,\zeta}(\lambda+\mu)_{\lambda+\mu-\gamma}$ can be written in the form $m=v^*_{\lambda+\mu}y$ for $y\in U_{K,\zeta}(\Gn)_{-\gamma}$.
For such $m$ we have 
\begin{align*}
&\langle m, \Xi({\nabla}_{K,\zeta}(\lambda)\otimes v_\mu)\rangle
=
\langle \Xi^*(m), {\nabla}_{K,\zeta}(\lambda)\otimes v_\mu)\rangle
=
\langle \Xi^*(v^*_{\lambda+\mu}y), {\nabla}_{K,\zeta}(\lambda)\otimes v_\mu)\rangle
\\
=&
\langle (v^*_{\lambda}\otimes v^*_{\mu})y
, {\nabla}_{K,\zeta}(\lambda)\otimes v_\mu)\rangle.
\end{align*}
By $\Delta(y)\in y\otimes k_{-\gamma}+\sum_{\delta\in Q^+\setminus\{0\}} U_{K,\zeta}(\Gn)\otimes U_{K,\zeta}(\Gh)U_{K,\zeta}(\Gn)_{-\delta}$ we have
\[
\langle m, \Xi({\nabla}_{K,\zeta}(\lambda)\otimes v_\mu)\rangle
=
\langle v^*_{\lambda}y, {\nabla}_{K,\zeta}(\lambda)\rangle.
\]
Hence $M=v^*_{\lambda+\mu}A$ with
$
A=
\{
y\in U_{K,\zeta}(\Gn)\mid v^*_{\lambda}y=0\}
$.
By \eqref{eq:UU-prime} we can also write
$M=v^*_{\lambda+\mu}A'$ with
$
A'=
\{
y\in {U}'_{K,\zeta}(\Gn)\mid \overline{v}^*_{\lambda}y=0\}
$.
On the other hand by a similar argument we have
$\overline{M}=\overline{v}^*_{\lambda+\mu}\overline{A}$ with
$
\overline{A}=
\{
y\in \overline{U}_K(\Gn)\mid \overline{v}^*_{\lambda}y=0\}
$.
Therefore, we obtain $\overline{M}=M$ from Lemma \ref{lem:u-prime}. 
The proof of Theorem \ref{thm:main2} is now complete in the case $\zeta_\alpha^2=1$ for any $\alpha\in\Delta$.

\subsection{}
Finally, we consider the case where the multiplicative order of $\zeta\in K^\times$ is finite.
Denote the multiplicative order of $\zeta$ by $\ell$, and
consider the root datum $({}^\sharp X,{}^\sharp \Delta, {}^\sharp Y,{}^\sharp \Delta^\vee)$ as in Section \ref{sect:QE}.
Recall that  we have an embedding 
\[
O_{K,\zeta}({}^\sharp N\backslash {}^\sharp G)
\subset 
O_{K,\zeta}(N\backslash G)
\] 
of algebras satisfying 
\[
O_{K,\zeta}({}^\sharp N\backslash {}^\sharp G;\lambda)
\subset
O_{K,\zeta}(N\backslash G;
\lambda)
\]
for $\lambda\in {}^\sharp X^+$.
Let $\mu\in X^+$.
We can take $\nu\in X^+$ and $\mu'\in{}^\sharp X^+$ such that $\mu+\nu=\mu'$.
By the result of the preceding subsection together with \eqref{eq:extremalA}, \eqref{eq:extremalB} we obtain
\[
O_{K,\zeta}({}^\sharp N\backslash {}^\sharp G;\lambda'+\mu')
=
\sum_{w\in W}
O_{K,\zeta}({}^\sharp N\backslash {}^\sharp G;\lambda')
\sigma^w_{\mu'}
\]
for some $\lambda'\in {}^\sharp X^+$.
Let $\xi\in X^+$  such that $\langle\xi,\alpha_i^\vee\rangle\gg0$ for any $i\in I$.
Then by  Proposition \ref{prop:XY} we have
\begin{align*}
&
O_{K,\zeta}(N\backslash G;\xi+\lambda'+\nu+\mu)
=O_{K,\zeta}(N\backslash G;\xi+\lambda'+\mu')
\\
=&
O_{K,\zeta}(N\backslash G;\xi)
O_{K,\zeta}({}^\sharp N\backslash {}^\sharp G;\lambda'+\mu')
=
\sum_{w\in W}
O_{K,\zeta}(N\backslash G;\xi)
O_{K,\zeta}({}^\sharp N\backslash {}^\sharp G;\lambda')
\sigma^w_{\mu'}
\\
\subset&
\sum_{w\in W}
O_{K,\zeta}(N\backslash G;\xi+\lambda')
\sigma^w_{\mu'}
=
\sum_{w\in W}
O_{K,\zeta}(N\backslash G;\xi+\lambda')
\sigma^w_{\nu}\sigma^w_{\mu}
\\
\subset&
\sum_{w\in W}
O_{K,\zeta}(N\backslash G;\xi+\lambda'+\nu)
\sigma^w_{\mu}.
\end{align*}
The proof of Theorem \ref{thm:main2} is complete.

\section*{acknowledgment}
A part of this work was done while the author was staying at East China Normal University in 2019 May
as a Zijiang Professor.
I would like to thank the members of the Department of Mathematics of East China Normal University, especially Bin Shu, for their hospitality.

\bibliographystyle{unsrt}

\end{document}